\documentclass[reqno, a4paper]{amsart}

\usepackage[T1]{fontenc}
\usepackage[british]{babel}
\usepackage{graphicx}
\usepackage{amsfonts}
\usepackage{amssymb}
\usepackage{amsthm}
\usepackage{amsmath}
\usepackage[all]{xy}
\usepackage{hyperref}
\usepackage{mathbbol}
\usepackage{mathabx}

\theoremstyle{plain}
\newtheorem{theorem}[subsection]{Theorem}
\newtheorem{lemma}[subsection]{Lemma}
\newtheorem{proposition}[subsection]{Proposition}

\theoremstyle{definition}
\newtheorem{definition}[subsection]{Definition}
\newtheorem{remark}[subsection]{Remark}
\newtheorem{notation}[subsection]{Notation}
\newtheorem{example}[subsection]{Example}
\newtheorem{examples}[subsection]{Examples}

\newcommand{\defn}{\textbf}

\newcommand{\comp}{\raisebox{0.2mm}{\ensuremath{\scriptstyle{\circ}}}}
\newcommand{\del}{\partial}
\renewcommand{\implies}{$\Rightarrow$}
\renewcommand{\equiv}{$\Leftrightarrow$}
\newcommand{\noproof}{\hfill \qed}

\newcommand{\normal}{\ensuremath{\triangleleft}}
\newcommand{\join}{\ensuremath{\vee}}
\newcommand{\meet}{\ensuremath{\wedge}}
\newcommand{\tensor}{\ensuremath{\otimes}}
\newcommand{\bigtensor}{\ensuremath{\bigotimes}}
\newcommand{\point}{\ensuremath{\rightleftarrows}}
\newcommand{\comma}{\downarrow}

\newcommand{\som}[2]{\bigl\langle\begin{smallmatrix} {#1} \\ {#2}\end{smallmatrix}\bigr\rangle}
\newcommand{\bigsom}[2]{\left\langle\begin{smallmatrix} {#1} \\ {#2}\end{smallmatrix}\right\rangle}
\newcommand{\somm}[3]{\left\langle\begin{smallmatrix} {#1} \\ {#2} \\ {#3}\end{smallmatrix}\right\rangle}

\newcommand{\coeq}{\ensuremath{\mathrm{coeq}}}
\newcommand{\Coeq}{\ensuremath{\mathrm{Coeq}}}

\renewcommand{\H}{\ensuremath{\mathrm{H}}}
\renewcommand{\ker}{\ensuremath{\mathrm{ker}}}
\newcommand{\Ker}{\ensuremath{\mathrm{Ker}}}
\renewcommand{\Im}{\ensuremath{\mathrm{Im}}}
\newcommand{\im}{\ensuremath{\mathrm{im}}}
\newcommand{\cre}{\ensuremath{\mathrm{cr}}}
\newcommand{\Huq}{\ensuremath{\mathrm{Huq}}}
\newcommand{\Smith}{\ensuremath{\mathrm{S}}}

\newcommand{\A}{\ensuremath{\mathcal{A}}}

\newcommand{\C}{\ensuremath{\mathcal{C}}}
\newcommand{\D}{\ensuremath{\mathcal{D}}}
\newcommand{\V}{\ensuremath{\mathcal{V}}}

\newcommand{\Ab}{\ensuremath{\mathsf{Ab}}}

\newcommand{\CAlg}{\ensuremath{\mathsf{CAlg}}}

\newcommand{\Gp}{\ensuremath{\mathsf{Gp}}}
\newcommand{\Loop}{\ensuremath{\mathsf{Loop}}}
\newcommand{\Mod}{\ensuremath{\mathsf{Mod}}}
\newcommand{\Mal}{\ensuremath{\mathsf{Mal}}}
\newcommand{\Pt}{\ensuremath{\mathsf{Pt}}}
\newcommand{\PXMod}{\ensuremath{\mathsf{PXMod}}}
\newcommand{\RG}{\ensuremath{\mathsf{RG}}}

\newcommand{\Set}{\ensuremath{\mathsf{Set}}}
\newcommand{\XMod}{\ensuremath{\mathsf{XMod}}}

\newcommand{\ab}{\ensuremath{\mathsf{ab}}}

\newcommand{\KK}{\ensuremath{K}}
\newcommand{\ZZ}{\ensuremath{\mathbb{Z}}}

\newdir{>>}{{}*!/3.5pt/:(1,-.2)@^{>}*!/3.5pt/:(1,+.2)@_{>}*!/7pt/:(1,-.2)@^{>}*!/7pt/:(1,+.2)@_{>}}
\newdir{ >>}{{}*!/8pt/@{|}*!/3.5pt/:(1,-.2)@^{>}*!/3.5pt/:(1,+.2)@_{>}}
\newdir{ |>}{{}*!/-3.5pt/@{|}*!/-8pt/:(1,-.2)@^{>}*!/-8pt/:(1,+.2)@_{>}}
\newdir{ >}{{}*!/-8pt/@{>}}
\newdir{>}{{}*:(1,-.2)@^{>}*:(1,+.2)@_{>}}
\newdir{<}{{}*:(1,+.2)@^{<}*:(1,-.2)@_{<}}

\newbox\pullbackbox
\setbox\pullbackbox=\hbox{\xy 0;<1mm,0mm>: \POS(4,0)\ar@{-} (0,0) \ar@{-} (4,4)
\endxy}
\def\pullback{\copy\pullbackbox}

\begin{document}

\author{Manfred Hartl}
\address{Universit\'e Lille Nord de France, F-59044~Lille, France\newline
UVHC, LAMAV and FR CNRS 2956, F-59313~Valenciennes, France}
\thanks{The first author wishes to thank the Centro de Matem\'atica da Universidade de Coimbra, the Institut de recherche en math\'ematique et physique, the Max-Planck-Institut f\"ur Mathematik and the Chern Institute for their kind hospitality during his stay at Coimbra, at Louvain-la-Neuve, at Bonn and at Tianjin}
\email{Manfred.Hartl@univ-valenciennes.fr}

\author{Tim Van~der Linden}
\address{CMUC, Universidade de Coimbra, 3001--454 Coimbra, Portugal\newline
Institut de recherche en math\'ematique et physique, Universit\'e catholique de Louvain, chemin du cyclotron~2 bte~L7.01.02, 1348 Louvain-la-Neuve, Belgium}
\thanks{The second author works as \emph{charg\'e de recherches} for Fonds de la Recherche Scientifique--FNRS. His research was supported by Centro de Matem\'atica da Universidade de Coimbra and by Funda\c c\~ao para a Ci\^encia e a Tecnologia (grant number SFRH/BPD/38797/2007). He wishes to thank the Max-Planck-Institut f\"ur Mathematik and UCT for their kind hospitality during his stay at Bonn and at Cape Town}
\email{tim.vanderlinden@uclouvain.be}

\title[The ternary commutator obstruction]{The ternary commutator obstruction\\ for internal crossed modules}



\date{\today}

\begin{abstract}
In finitely cocomplete homological categories, co-smash products give rise to (possibly higher-order) commutators of subobjects. We use binary and ternary co-smash products and the associated commutators to give characterisations of internal crossed modules and internal categories, respectively. The ternary terms are redundant if the category has the \emph{Smith is Huq} property, which means that two equivalence relations on a given object commute precisely when their normalisations do. In fact, we show that the difference between the Smith commutator of such relations and the Huq commutator of their normalisations is measured by a ternary commutator, so that the \emph{Smith is Huq} property itself can be characterised by the relation between the latter two commutators. This allows to show that the category of loops does not have the \emph{Smith is Huq} property, which also implies that ternary commutators are generally not decomposable into nested binary ones. 

Thus, in contexts where \emph{Smith is Huq} need not hold, we obtain a new description of internal categories, Beck modules and double central extensions, as well as a decomposition formula for the Smith commutator. The ternary commutator now also appears in the Hopf formula for the third homology with coefficients in the abelianisation functor. 
\end{abstract}

\maketitle

\section*{Introduction}
Internal crossed modules in a semi-abelian category~\cite{Janelidze-Marki-Tholen} were introduced by \mbox{Janelidze} in the article~\cite{Janelidze}. His starting point is the desired correspondence between crossed modules and internal categories, which determines the basic properties that crossed modules should satisfy. His definition is based on the concept of internal action which he introduced with Bourn in~\cite{Bourn-Janelidze:Semidirect} and which is further worked out in the article~\cite{BJK}. 

He explains that the extension of the case of groups to semi-abelian categories is not entirely without difficulties. The most straightforward description of the concept of crossed module merely gives so-called \emph{star-multiplicative graphs}---in which the composition of morphisms is only defined locally around the origin---and not the internal groupoids one would expect, in which every composable pair of morphisms can actually be composed. This defect can be mended, as it is indeed done in~\cite{Janelidze}. Unfortunately, the resulting characterisation of internal crossed modules becomes slightly more complicated than expected after considering the groups case. 

This gave rise to the question, whether every star-multiplicative graph can be equipped with a unique internal groupoid structure. It turns out~\cite{MFVdL} that the gap between the two is precisely as big as the gap between the Huq commutator of normal subobjects and the Smith commutator of internal equivalence relations. That is to say, in a semi-abelian category they are equivalent if and only if the \emph{Smith is Huq} condition holds. This explains why the difference between the two concepts is invisible in the category of groups, in fact in any of the concrete algebraic categories where internal crossed modules were ever studied: all of those are strongly protomodular (and action accessible), which as we know implies the \emph{Smith is Huq} condition. 

Introducing ternary commutators gives a different view on the situation, more natural in a sense: just as internal groupoids can be described as internal reflexive graphs with a certain binary (Smith) commutator being trivial, now we can say that internal groupoids may also be described as internal star-multiplicative graphs for which a certain ternary (Higgins) commutator is zero. Equivalently, a certain coherence condition involving ternary co-smash products holds for the associated internal precrossed module (satisfying the Peiffer condition).
A~byproduct of this analysis is that the context is enlarged to a non-exact setting (being careful with the notion of star-multiplicativity), as we may mostly work in finitely cocomplete homological categories instead of semi-abelian ones.

\subsection*{Internal actions}
It is well known that \emph{every split epimorphism of groups is a semi-direct product projection}. This fact gives rise to an equivalence between the category $\Pt_{B}(\Gp)$ of split epimorphisms of groups (with chosen splitting) with codomain $B$ and the category of $B$-groups. Similarly~\cite{Bourn-Janelidze:Semidirect, BJK}, in a semi-abelian category~$\A$, internal actions correspond to split epimorphisms. Furthermore, since the kernel functor~${\Pt_{B}(\A)\to \A}$ is monadic for every object $B$ in $\A$, the internal actions are defined as the algebras over the corresponding monad.

The right adjoint ${\A\to \Pt_{B}(\A)}$ sends an object $X$ of $\A$ to the epimorphism $\som{1_{B}}{0}\colon {B+X\to B}$ split by $\iota_{B}\colon{B\to B+X}$. As a consequence, the induced monad $(B\flat(-),\mu^{B},\eta^{B})$ on $\A$ is defined, as a functor, by
\[
B\flat X=\Ker\bigl(\som{1_{B}}{0}\colon {B+X\to B}\bigr).
\]
Hence a $B$-action on $X$ is a morphism $\xi\colon{B\flat X\to X}$ satisfying the algebra axioms, and any such $\xi$ corresponds to a split epimorphism ${X\ltimes_{\xi}B\point B}$.

In contrast to the presentation in~\cite{Bourn-Janelidze:Semidirect, BJK} we are interested in replacing $B\flat X$ with the object $B\tensor X$ which can be defined via the split short exact sequence 
\begin{equation}\label{tensor vs bemol}
\vcenter{\xymatrix{0 \ar[r] & B\tensor X \ar@{.>}[rd]_-{\psi} \ar@{{ |>}->}[r] & B\flat X \ar@{.>}[d]_-{\xi} \ar@<.5ex>@{-{ >>}}[r] & X \ar@{.>}[ld]^-{1_{X}} \ar@<.5ex>@{{ >}->}[l]^(.55){\eta_{X}^{B}} \ar[r] & 0\\
&& X }}
\end{equation}
(of solid arrows). The dotted arrow $\xi$ in this diagram represents an action in the sense of~\cite{Bourn-Janelidze:Semidirect, BJK}, while the dotted arrow $\psi$ is the morphism we are interested in replacing $\xi$ with. Assuming that diagram~\eqref{tensor vs bemol} commutes, observing that $\xi$ and $1_{X}$ are jointly epic, we can say that $\xi$ and $\psi$ determine each other. We think of $\psi$ as a ``fragment'' of $\xi$ which, however, determines it~\cite{Actions}.

Our strategy is to characterise internal crossed modules through such fragments of actions ${B\tensor X\to X}$ instead of the morphisms ${B\flat X\to X}$ which determine them. Thus we shall actually be considering the algebras of the endofunctor $B\tensor(-)$ rather than those of the monad $(B\flat(-),\mu^{B},\eta^{B})$.

The disadvantage of this approach is that we have to keep track of which ${B\tensor(-)}$-algebras, that is, which morphisms ${B\tensor X\to X}$, do determine an action---especially since we shall be working in a non-exact context where, as explained in~\cite{MFS}, only certain $(B\flat(-),\mu^{B},\eta^{B})$-algebras will correspond to split epimorphisms. The advantage is that the object $B\tensor X$ turns out to be a \emph{co-smash product}, and these can be used to define (higher-order) commutators:

\subsection*{Commutators via co-smash products}
Here the basic idea---which was discovered independently in~\cite{Actions} and~\cite{MM-NC}---is that the \defn{co-smash product} or \defn{tensor product}~\cite{Smash}
\[
K\tensor L=\Ker\bigl(\bigl\langle\begin{smallmatrix} 1_{K} & 0 \\
0 & 1_{L}\end{smallmatrix}\bigr\rangle
\colon K+L\to K\times L\bigr)
\]
of objects $K$, $L$ in a finitely cocomplete homological category behaves as a kind of ``formal commutator'' of~$K$ and~$L$. If now $k\colon K \to X$ and $l\colon L \to X$ are subobjects of an object $X$, then their \defn{(Higgins) commutator} $[K,L]\leq X$ is the image of the induced morphism
\[
\xymatrix@=2em{K\tensor L \ar@{{ |>}->}[r]^-{\iota_{K,L}} & K+L \ar[r]^-{\bigsom{k}{l}} & X.}
\]
Using higher co-smash products it is easy to extend this definition to higher-order commutators: for instance, given a third subobject $m\colon M \to X$ of $X$, the ternary commutator $[K,L,M]\leq X$ is the image of the composite
\[
\xymatrix@=3em{K\tensor L\tensor M \ar@{{ |>}->}[r]^-{\iota_{K,L,M}} & K+L+M \ar[r]^-{\somm{k}{l}{m}} & X,}
\]
where $\iota_{K,L,M}$ is the kernel of
\[
\xymatrix@=8em{K+L+M \ar[r]^-{\left\langle\begin{smallmatrix} \iota_{K} & \iota_{K} & 0 \\
\iota_{L} & 0 & \iota_{L}\\
0 & \iota_{M} & \iota_{M}\end{smallmatrix}\right\rangle} & (K+L)\times (K+M) \times (L+M).}
\]
The basic properties of the (binary) Higgins commutator are explored in the articles~\cite{Actions} and~\cite{MM-NC}. In the former it is also explained how this commutator is related to internal actions. We shall recall some of this in sections~\ref{Section-Co-Smash} and~\ref{Section-Semi-Direct-Products}.

\subsection*{The ternary commutator obstruction}
One of the main results of the present paper is that the \emph{Smith is Huq} condition for finitely cocomplete homological categories may be expressed in terms of co-smash products as the vanishing of a ternary commutator. 

Indeed, we prove that for equivalence relations $R$ and $S$ on $X$ with normalisations $K$, $L\normal X$, respectively, the relations $R$ and $S$ centralise each other in the sense of Smith~\cite{Smith} if and only if the commutators $[K,L]$ and $[K,L,X]$ are trivial. Since $[K,L]=0$ precisely when $K$ and $L$ commute in the sense of Huq~\cite{Huq}, the object~$[K,L,X]$ is the \defn{ternary commutator obstruction} for the \emph{Smith is Huq} condition, that is, $[K,L,X]\leq [K,L]^{\Huq}$ precisely when the Huq commutator is the normalisation of the Smith commutator.

The fact that in the category of groups and in other familiar algebraic categories the \emph{Smith is Huq} condition holds, agrees with the fact that here all ternary commutators are expressible in terms of binary ones. In general, though, ternary commutators cannot be written in terms of repeated binary commutators.

This new viewpoint on the \emph{Smith is Huq} condition gives new examples of categories which satisfy it. A \emph{nilpotent category of class $2$} is a semi-abelian category whose identity functor is \emph{quadratic}, which means that it has a trivial ternary co-smash product~\cite{CCC}. Hence, almost by definition, any such category satisfies \emph{Smith is Huq}. In particular, the \emph{Smith is Huq} condition holds for modules over a square ring, and specifically for algebras over a nilpotent operad of class two~\cite{BHP}.

On the other hand, the category of loops (quasigroups with an identity) does not satisfy \emph{Smith is Huq}: we give an example of a loop $X$ with an abelian subloop $A$ and elements $a\in A$, $x\in X$ such that the associator element $\ldbrack a,a,x\rdbrack$ is non-trivial. (In fact, one of the first examples of a non-associative structure ever considered will do, see Example~\ref{Example-Loops}. Freese and McKenzie give another example of this situation in their book~\cite{Freese-McKenzie}.) This proves that the ternary commutator $[A,A,X]$, which contains all such associator elements, need not vanish even when the binary commutator $[A,A]$ does. As a consequence, $\Loop$ is not action accessible or strongly protomodular---though it is well known to be semi-abelian~\cite{Borceux-Clementino}.

\subsection*{Characterisation of internal crossed modules}
We shall consider quadruples $(G,A,\mu ,\del)$ in which $G$ and $A$ are objects, $\mu \colon{A\tensor G\to A}$ determines an action of~$G$ on~$A$, and $\del\colon{A \to G}$ is a morphism. We prove that such a quadruple is a crossed module if and only if the following three diagrams commute.
\[
{\xymatrix{
A\tensor G \ar[r]^-\mu \ar[d]_{\del\tensor 1_{G}} & A\ar[d]^{\del}\\
G\tensor G\ar[r]_-{c^{G,G}} & G}}
\qquad
{\xymatrix{
A\tensor A \ar[r]^-{c^{A,A}} \ar[d]_-{1_{A}\tensor \del} & A \ar@{=}[d] \\
A\tensor G \ar[r]_-{\mu} & A}}
\qquad
{\xymatrix{
A\tensor A\tensor G \ar[r]^-{\mu_{2,1}} \ar[d]_-{1_{A}\tensor \del\tensor 1_{G}} & A \ar@{=}[d] \\
A\tensor G\tensor G \ar[r]_-{\mu_{1,2}} & A}}
\]
The first diagram expresses the \emph{precrossed module condition} which says that the morphism~$\del$ is $G$-equivariant with respect to $\mu$ and the conjugation action~$c^{G,G}$ of~$G$ on itself~\cite{Janelidze-Marki-Ursini}. Quadruples which satisfy this first condition correspond to internal reflexive graphs. Commutativity of the middle diagram is the so-called \emph{Peiffer condition}: the conjugation action $c^{A,A}$ of $A$ on itself coincides with the pullback~$\del^{*}(\mu)$ of~$\mu$ along~$\del$. Quadruples which satisfy the first two conditions correspond to  \emph{Peiffer graphs} in the sense of~\cite{MM}, which admit some kind of composition locally around the origin, and which are equivalent to having a star-multiplicative graph structure (\cite{Janelidze, MM}, see also~\cite{MFVdL}). The diagram on the right commutes when the local composition of the star-multiplication extends to a globally defined internal groupoid structure.

\subsection*{Internal categories in a homological category}
Our analysis of internal crossed modules depends on a new characterisation of internal categories in terms of commutators, valid in any finitely cocomplete homological category. Let us just mention here that an internal reflexive graph
\[
\xymatrix@!0@=5em{R \ar@<1ex>[r]^-{d} \ar@<-1ex>[r]_-{c} & G \ar[l]|-{e}}
\qquad\qquad
d\comp e = c\comp e = 1_{G}
\]
will be an internal category when either of the following equivalent conditions holds (Theorem~\ref{Theorem-Internal-Categories}):
\begin{itemize}
\item $[\Ker(d),\Ker(c)]=0=[\Ker(d),\Ker(c),R]$;
\item $[\Ker(d),\Ker(c)]=0=[\Ker(d),\Ker(c),\Im(e)]$;
\item the morphism $c^{A,R}\colon{A\tensor R\to A}$ induced by the conjugation action of $R$ on~$A=\Ker(d)$ factors through $1_{A}\tensor c\colon{A\tensor R\to A\tensor G}$;
\item $c^{A,R}=(e\comp c)^{*}(c^{A,R})$.
\end{itemize}

\subsection*{Beck modules}
As a special case we find a new characterisation of the concept of \emph{Beck module}---which, via~\cite{Bourn-Janelidze:Semidirect, Bourn-Janelidze:Torsors}, is the same things as an abelian action---in terms of tensor products: a $G$-action on an abelian object $A$ determined by a morphism $\psi\colon{A\tensor G\to A}$ is a $G$-module structure on $A$ if and only if a certain induced morphism $\psi_{2,1}\colon{A\tensor A\tensor G\to A}$ is trivial.

\subsection*{An application in homology}
We give a concrete application of these results in semi-abelian homology. First we characterise double central extensions~\cite{EGVdL, Janelidze:Double, Janelidze-Kelly, RVdL} in terms of binary and ternary commutators, and then we apply the main result of~\cite{EGVdL} obtain a Hopf formula for the third homology of an object $Z$ with coefficients in the abelianisation functor: 
\[
\H_{3}(Z,\ab)\cong\frac{K\meet L\meet [X,X]}{[K,L,X]\join [K,L]\join [K\meet L,X]}
\]
where $K$, $L\normal X$ are the kernels induced by a double presentation of $Z$. This formula is valid in any semi-abelian category with enough projectives, whether the \emph{Smith is Huq} condition holds or not.

\subsection*{Structure of the text}
In Section~\ref{Section-Categorical-Context} we sketch the categorical context in which we shall be working. Section~\ref{Section-Co-Smash} is devoted to co-smash products and (higher-order) commutators. Section~\ref{Section-Semi-Direct-Products} discusses semi-direct products. In Section~\ref{Section-SH} we give a characterisation of the \emph{Smith is Huq} condition in terms of ternary commutators---Theorem~\ref{Theorem-SH}, the key result of the paper---and a formula for the Smith commutator of equivalence relations in terms of a binary and a ternary commutator of normal subobjects (Theorem~\ref{Theorem-Smith-Commutator}). We also find a characterisation of double central extensions (Proposition~\ref{Characterisation-Double-Central-Extensions}) which yields an explicit version of the Hopf formula for the third homology of an object (Theorem~\ref{Theorem-H3}). This leads to Section~\ref{Section-Crossed-Modules} where we give new characterisations of internal categories and internal crossed modules (Theorem~\ref{Theorem-Internal-Categories}, Theorem~\ref{Theorem-Characterisation-XMod}). In Section~\ref{Section-Beck} we characterise Beck modules in similar terms (Theorems~\ref{Theorem-Beck}, \ref{Theorem-Beck-2}).

\section{The categorical context}\label{Section-Categorical-Context}

\subsection{Pointed categories}
A \defn{pointed} category is a category with a \defn{zero object}, that is, an object which is at the same time initial and terminal.

\subsection{Regular and exact categories}
Recall that a \defn{regular epimorphism} is the coequaliser of some parallel pair of morphisms. A \defn{regular} category is a finitely complete category having a pullback-stable (regular epi, mono)-factorisation system. Given a morphism $f\colon{X\to Y}$, we write $\im(f)\colon {\Im(f)\to Y}$ for the mono-part in this \defn{image factorisation} of $f$. If $M\leq X$ is a subobject of~$X$ then we write $f(M)$ for the \defn{direct image} of $M$ along~$f$: it is the image of $f\comp m$, where~$m\colon{M\to X}$ is a monomorphism that represents the subobject.

Regular categories provide a natural context for working with relations. We denote the kernel relation (=~kernel pair) of a morphism $f\colon{X\to Y}$ (that is, the pullback of~$f$ along itself) by $(X\times_{Y}X,f_{1},f_{2})$. A regular category is said to be \defn{Barr exact} when every equivalence relation is \defn{effective}, which means that it is the kernel pair of some morphism~\cite{Barr}.

\subsection{Homological and semi-abelian categories}
A pointed category with pullbacks is called \defn{protomodular}~\cite{Bourn1991} if and only if the Split Short Five Lemma holds. If, moreover, the pointed category is regular, then protomodularity is equivalent to the (Regular) Short Five Lemma. This means that, given a commutative diagram
\begin{equation*}\label{Short-Five-Lemma}
\vcenter{\xymatrix{0 \ar[r] & A' \ar@{{ |>}->}[r] \ar[d]_-a & X' \ar@{-{>>}}[r]^-{p'} \ar[d]_-x & G' \ar[d]^-g\\
0 \ar[r] & A \ar@{{ |>}->}[r] & X \ar@{-{>>}}[r]_-p & G}}
\end{equation*}
with regular epimorphisms $p$, $p'$ and their kernels, if $a$ and $g$ are isomorphisms then also $x$ is an isomorphism. We usually denote the kernel of a morphism~$f$ by~$(\Ker(f), \ker(f))$, and say that a morphism is \defn{proper} when its image is a \defn{normal} monomorphism (= a kernel). If~$M\leq X$ is a normal subobject then we write~${M\normal X}$.

A \defn{homological} category~\cite{Borceux-Bourn} is a category which is pointed, regular and protomodular. This is a context where many of the basic diagram lemmas of homological algebra hold. In particular, here the notion of \defn{(short) exact sequence} has its full meaning: it is a regular (hence, in this context, normal) epimorphism with its kernel such as
\begin{equation*}\label{ses}
\xymatrix{0 \ar[r] & A \ar@{{ |>}->}[r]^-{a} & X \ar@{-{ >>}}[r]^-{p} & G \ar[r] & 0.}
\end{equation*}
This short exact sequence is called \defn{split} when there exists a \defn{section} (or \defn{splitting}) $s\colon G\to X$ of~$p$, that is, a morphism $s$ such that $p\comp s=1_{G}$.

Note that a split epimorphism $p\colon{X\to G}$ may have many splittings. When just one splitting $s$ is chosen, the pair $(p,s)$ is called a \defn{point (over $G$)}. The \defn{category of points $\Pt(\A)$} is defined by taking points in $\A$ (considered as diagrams $p\comp s=1_{G}$) as objects and natural transformations between points as morphisms. The points over a given object $G$ form the full subcategory $\Pt_{G}(\A)=(1_{G}\comma(\A\comma G))$ of $\Pt(\A)$.

In a finitely cocomplete homological category, any comparison morphism
\[
\left\langle\begin{smallmatrix} 1_{X} & 0 \\
0 & 1_{Y}
\end{smallmatrix}\right\rangle\colon X+Y\to X\times Y
\]
is a regular epimorphism. 

A \defn{Mal'tsev} category~\cite{Carboni-Lambek-Pedicchio} is by definition a finitely complete category in which every reflexive relation is necessarily an equivalence relation. It is well known that any finitely complete protomodular category satisfies this property~\cite{Bourn1996}. Furthermore, the Mal'tsev property is preserved by slicing. This is a context in which many of the basic constructions in commutator theory make sense. In a Mal'tsev category, internal categories are automatically internal groupoids.

A \defn{semi-abelian} category is a homological category which is exact and has binary sums~\cite{Janelidze-Marki-Tholen}. In a semi-abelian category, the direct image of a kernel along a regular epimorphism is still a kernel. In this context, the existence of binary sums entails finite cocompleteness.

We shall always work in a finitely cocomplete homological category~\cite{Borceux-Bourn} $\A$ unless explicitly mentioned otherwise. Some proofs need a semi-abelian~\cite{Janelidze-Marki-Tholen} environment; we always explain where and why.

\section{Co-smash products and commutators}\label{Section-Co-Smash}
We explain how co-smash products~\cite{Smash} in a finitely cocomplete homological category give rise to (higher-order) commutators. We start with some basic definitions and properties, we give some examples and recall how the binary commutator is a categorical version of the Higgins commutator~\cite{Actions, Higgins, MM-NC}.

\subsection{Notations for sums and products}
In a pointed category with finite sums, we denote the coproduct inclusion~${X_k\to X_1+\cdots +X_n}$ by $\iota_{X_k}$ or by $\iota_k$, and its canonical retraction $X_1+\cdots +X_n\to X_k$ by $\rho_{X_k}$ or $\rho_k$.

Dually, when working in a pointed category with finite products, we denote the product projection $X_1\times \cdots \times X_n\to X_k$ by $\pi_{X_{k}}$ or $\pi_k$ and its canonical section 
\[
\langle 0, \dots, 1_{X_{k}},\dots , 0\rangle\colon X_k\to X_1\times \cdots \times X_n
\]
by $\sigma_{X_{k}}$ or $\sigma_k$.

\begin{definition}\cite{Smash}\label{co-smash product}
In a finitely complete and cocomplete pointed category $\A$, we call \defn{co-smash product} or \defn{tensor product} $\bigtensor_{k=1}^{n}X_{k}=X_{1}\tensor \cdots \tensor X_{n}$ of objects $X_{1}$, \dots, $X_{n}$, $n\geq 2$ the kernel
\[
\xymatrix{{\displaystyle\bigtensor_{k=1}^{n}X_{k}} \ar@{{ |>}->}[r] & {\displaystyle\coprod_{k=1}^{n}X_{k}} \ar[r]^-{r} & {\displaystyle\prod_{k=1}^{n}\coprod_{j\neq k}X_{j}}} 
\]
where $r$ is the comparison morphism determined by
\[
\pi_{\coprod_{j\neq m}X_{j}}\comp r\comp \iota_{X_{l}}=
\begin{cases}
\iota_{X_{l}} & \text{if $l\neq m$}\\
0 & \text{if $l= m$}
\end{cases}
\]
for $l$, $m\in \{1,\dots, n\}$. The kernel morphism is usually denoted $\iota_{X_{1},\dots,X_{n}}$.
\end{definition}

We shall only consider co-smash products in situations where $\A$ is at least finitely cocomplete homological.

\begin{example}
Let us make explicit what happens in the lowest-dimensional cases, which are essential in the present article. If $n=2$ then we obtain a short exact sequence
\[
\xymatrix@=3em{0\ar[r]& X\tensor Y \ar@{{ |>}->}[r]^-{\iota_{X,Y}} & X+Y \ar@{-{ >>}}[r]^-{\left\langle\begin{smallmatrix} 1_{X} & 0 \\
0 & 1_{Y}
\end{smallmatrix}\right\rangle} & X\times Y \ar[r] & 0}
\]
for any $X$, $Y$ in $\A$. Note that the object $X\tensor Y$ is denoted $X\diamond Y$ in the article~\cite{MM-NC}. If $n=3$ and $X$, $Y$, $Z$ are objects of $\A$, then we consider the morphism
\[
\xymatrix@=8em{X+Y+Z \ar[r]^-{\left\langle\begin{smallmatrix} \iota_{X} & \iota_{X} & 0 \\
\iota_{Y} & 0 & \iota_{Y}\\
0 & \iota_{Z} & \iota_{Z}\end{smallmatrix}\right\rangle} & (X+Y)\times (X+Z) \times (Y+Z),}
\]
which need no longer be a regular epimorphism; the co-smash product $X\tensor Y\tensor Z$ is its kernel.
\end{example}

\begin{example}\label{Example-Binary-Groups}
In the case of groups we have 
\[
X\tensor Y=\langle [x,y]\mid x\in X,y\in Y\rangle
\]
where $[x,y]=xyx^{-1}y^{-1}$. So $X\tensor Y$ is a kind of ``formal commutator'' of $X$ and $Y$ as explained in~\cite{MM-NC} and~\cite{Actions}. This fact gives rise to the definition of commutators in terms of co-smash products.

Given groups $X$, $Y$ and $Z$ with chosen elements $x$, $y$ and $z$, respectively, the ternary commutator word
\[
xyx^{-1}y^{-1}zyxy^{-1}x^{-1}z^{-1}=[[x,y],z]
\]
is an example of an element of $X\tensor Y\tensor Z$. 
\end{example}

\begin{example}\cite{Smash}\label{Example-Algebras}
Let $\KK$ be a commutative ring with unit and consider the category $\CAlg_{\KK}$ of non-unitary commutative $\KK$-algebras. Here the co-smash product $X\tensor Y$ is the tensor product $X\tensor_{\KK}Y$ over $\KK$.
\end{example}

\begin{example}\label{Example-Varieties}
In a pointed variety of algebras $\V$, an element of a sum $X+Y+Z$ is of the shape
\[
t(x_{1},\dots,x_{k},y_{1},\dots,y_{l},z_{1},\dots,z_{m})
\]
where $t$ is a term of arity $k+l+m$ in the theory of $\V$ and $x_{1}$, \dots, $x_{k}\in X$, $y_{1}$, \dots, $y_{l}\in Y$ and $z_{1}$, \dots, $z_{m}\in Z$. This element belongs to the co-smash product $X\tensor Y\tensor Z$ if and only if
\[
\begin{cases}
t(x_{1},\dots,x_{k},y_{1},\dots,y_{l},0,\dots,0)=0 & \text{in $X+Y$},\\
t(x_{1},\dots,x_{k},0,\dots,0,z_{1},\dots,z_{m})=0 & \text{in $X+Z$},\\
t(0,\dots,0,y_{1},\dots,y_{l},z_{1},\dots,z_{m})=0 & \text{in $Y+Z$}.
\end{cases}
\]
Here $0$ denotes the unique constant of the theory of $\V$.
\end{example}

\begin{remark}
It follows easily from the definitions that the sequence~\eqref{tensor vs bemol} is exact.
\end{remark}

\begin{remark}\label{Remark-Ternary-Cross-Effect}
A tensor product $X\tensor Y\tensor Z$ in $\A$ may be obtained as a \emph{cross-effect} of the functor $X\tensor (-)\colon{\A\to\A}$, evaluated in the pair $(Y,Z)$. This yields an alternative (inductive) definition of co-smash products, which allows for different proof techniques and also a different intuition. 

The concept of cross-effect of a functor between abelian categories was introduced by Eilenberg and Mac\,Lane in the article~\cite{Eilenberg-MacLane}, where it was used in the study of polynomial functors. This definition does, however, not generalise to non-additive contexts. The approach due to Baues and Pirashvili~\cite{Baues-Pirashvili}, worked out in the case of groups, does extend easily to more general situations. Let us briefly recall from~\cite{Actions, Hartl-Vespa} how. 

Let $F\colon {\C} \to {\D}$ be a functor from a pointed category with finite sums ${\C}$ to a pointed finitely complete category ${\D}$. The \defn{$n$-th cross-effect of~$F$} is the functor
\[
\cre_n(F)\colon \C^{n} \to \D
\]
defined by $\cre_1(F)(X) = \Ker\bigl(F(0)\colon F(X) \to F(0)\bigr)$ and, for $n>1$,
\[
\cre_n(F)(X_1,\ldots,X_n)=\Ker(r_{F}),
\]
with
\[
r_{F}\colon F(X_1 + \dots + X_n) \to \prod_{k=1}^n F(X_1 + \dots + \widehat{X_k} + \dots + X_n)
\]
as in Definition~\ref{co-smash product}, modulo the $F$. The usual notation for cross-effects is
\[
F(X_1|\cdots|X_n)=\cre_n(F)(X_1,\ldots,X_n).
\]
When $F$ is the identity functor $1_{\A}$ of $\A$ we clearly find the co-smash product
\[
X_1\tensor\cdots\tensor X_n = 1_{\A}(X_1|\cdots|X_n).
\]

Coming back to the claim made at the beginning of this remark: writing down the relevant $3\times 3$ diagram, it is easy to check that indeed, $X\tensor Y\tensor Z$ may be obtained as the second cross-effect $(X|-)(Y|Z)$ of the functor $(X|-)=X\tensor (-)$ evaluated in the pair of objects $(Y,Z)$.
\end{remark}

This principle may be used in the proof of the following result.

\begin{proposition}\label{cr-reg-epi} 
Co-smash products preserve regular epimorphisms: for instance, so do the functors $X\tensor(-)\colon {\A\to \A}$ and $(-)\tensor Y\tensor Z\colon{\A\to \A}$.
\end{proposition} 
\begin{proof}
This was proved for binary co-smash products in~\cite{MM-NC} and extended to binary cross-effects (of functors which preserve regular epimorphisms and the zero object) in~\cite{Actions}. The case of ternary co-smash products now follows. For higher-degree co-smash products a proof is obtained via a similar argument and induction on the degree.
\end{proof}

\begin{remark}\label{Remark-Associativity}
As explained in~\cite{Smash} and~\cite{CCC, Actions}, tensor products need not be associative. Nevertheless there are always comparison morphisms
\[
(X_{1}\tensor \cdots \tensor X_{i})\tensor X_{i+1}\tensor\cdots\tensor X_{n} \to \bigtensor_{k=1}^{n}X_{k}
\]
induced by bracketing inside a co-smash product. When $n=3$, for instance, we find the dotted arrow between the kernels in the diagram
\[
\xymatrix{0 \ar[r] & (X\tensor Y)\tensor Z \ar@{.>}[d] \ar@{{ |>}->}[r]^-{\iota_{X\tensor Y,Z}} & (X\tensor Y) + Z \ar[r] \ar[d]^-{\iota_{X,Y}+ 1_{Z}} & (X\tensor Y)\times Z \ar[d]^{\iota_{X,Y}\times \langle \iota_{Z}, \iota_{Z}\rangle} \\
0 \ar[r] & X\tensor Y\tensor Z \ar@{{ |>}->}[r]_-{\iota_{X,Y,Z}} & X+Y+Z \ar[r] & (X+Y)\times (X+Z)\times (Y+Z).}
\]
\end{remark}

\subsection{Joins and the sum decomposition}
For subobjects
\[
\xymatrix{L\ar@{{ >}->}[r]^-l & X & M \ar@{{ >}->}[l]_-m}
\]
of an object $X$ in a finitely cocomplete homological category we write
\[
L \join M=\Im \bigl(\som{l}{m}\colon L+M\to X\bigr).
\]
If $L\meet M=0$ and $L$ is normal in $L \join M$ then we write $L \join M=L\rtimes M$. Note that this occurs precisely when there is a split short exact sequence
\[\label{rtimessequ}
\xymatrix@=3em{0 \ar[r] & L \ar@{{ |>}->}[r]^-{l} & L \join M \ar@{-{ >>}}@<.5ex>[r] & M \ar[r] \ar@{{ >}->}@<.5ex>[l]^-{m} & 0,}
\]
which justifies the semi-direct product notation (see Section~\ref{Section-Semi-Direct-Products}). Like for the sum, morphisms defined on $L\rtimes M$ are completely determined by the effect on~$L$ and~$M$, so we write them in a column.

The following result is crucial: it gives us a formula which expresses the commutator of a join as a join of commutators (Proposition~\ref{Proposition-Join}), which will in turn be used to decompose complicated commutators into less complicated ones.

\begin{lemma}\label{Tensor vs sum}\cite{Actions}
Suppose that $\A$ is a finitely cocomplete homological category and $X$, $Y$ and $Z$ are objects of $\A$. Then we have a decomposition
\[
X\tensor(Y+Z)=\bigl((X\tensor Y\tensor Z)\rtimes (X\tensor Y)\bigr)\rtimes (X\tensor Z).
\]
We write $\iota_{Y,Z}^{X\tensor (-)}\colon{X\tensor Y\tensor Z\to X\tensor(Y+Z)}$ for the canonical inclusion.
\end{lemma}
\begin{proof}
In other words, we have split short exact sequences
\[
\xymatrix@=3em{0 \ar[r] & W \ar@{{ |>}->}[r]^-{k} & X\tensor(Y+Z) \ar@{-{ >>}}@<.5ex>[r]^-{1_{X}\tensor \rho_{Z}} & X\tensor Z \ar[r] \ar@{{ >}->}@<.5ex>[l]^-{1_{X}\tensor \iota_{Z}} & 0}
\]
and
\[
\xymatrix@=4em{0 \ar[r] & (X\tensor Y\tensor Z) \ar@{{ |>}->}[r]^-{l} & W \ar@{-{ >>}}@<.5ex>[r]^-{(1_{X}\tensor \rho_{Y})\circ k} & X\tensor Y \ar[r] \ar@{{ >}->}@<.5ex>[l]^-{m} & 0.}
\]
There are only two things to be shown: that there is indeed a splitting $m$ for the morphism $(1_{X}\tensor \rho_{Y})\comp k$, and that $l$ is the kernel of this split epimorphism $(1_{X}\tensor \rho_{Y})\comp k$. First of all, the morphism $m$ is the factorisation of
\[
1_{X}\tensor \iota_{Y}\colon X\tensor Y\to X\tensor (Y+Z)
\]
over $k$. Secondly, $k\comp l$ is the intersection of $k$ and the kernel of
\[
1_{X}\tensor \rho_{Y}\colon{X\tensor (Y+Z)\to X\tensor Y},
\]
so it is $X\tensor Y\tensor Z$ as explained in Remark~\ref{Remark-Ternary-Cross-Effect}. Note that $\iota_{Y,Z}^{X\tensor (-)}=k\comp l$.
\end{proof}

\subsection{Co-smash products induce higher-order commutators}\label{Subsection-Commutators}
We obtain the following categorical notion of commutator (of arbitrary length $n$) which was first introduced in~\cite{MM-NC} for $n=2$ and in~\cite{Actions} for all $n\geq 2$. It is more thoroughly studied in~\cite{CCC}.

\begin{definition}\label{comdef} 
Let $X$ be an object of a finitely cocomplete homological category. The \defn{$n$-fold commutator morphism} of $X$ is the composite morphism
\[
c_n^X\colon\xymatrix@1@=3em{X\tensor \cdots\tensor X \ar@{{ |>}->}[r]^-{\iota_{X,\dots,X}} & X+\cdots+X \ar[r]^-{\nabla_{X}} & X.}
\]
When $x_i\colon X_i \to X$ for $1\leq i\leq n$ are subobjects of $X$, their \defn{commutator} is the subobject
\begin{align*}
[X_1,\ldots,X_n] &= \Im \bigl(\xymatrix@1@=4em{X_{1}\tensor \cdots\tensor X_{n} \ar[r]^-{x_1\tensor\cdots\tensor x_n} & X\tensor \cdots\tensor X \ar[r]^-{c_n^X} & X}\bigr)\\
&= \Im \bigl(\xymatrix@1@=4em{X_{1}\tensor \cdots\tensor X_{n} \ar[r]^-{\iota_{X_{1},\dots,X_{n}}} & X_{1}+\cdots+X_{n} \ar[r]^-{\left\langle\begin{smallmatrix} x_{1} \\ \vdots \\\\ x_{n}\end{smallmatrix}\right\rangle} & X}\bigr)
\end{align*}
of $X$.
\end{definition}

\begin{example}\label{Group-Commutators}
In~\cite{CCC} the $n$-fold commutator $[X,\ldots,X] =\Im(c_n^X)$ is determined for the categories of groups and of loops. In the former, this term coincides with the $n$-th term of the lower central series; for loops, however, this is not true: it here coincides with the $n$-th term of the commutator-associator-filtration recently introduced by Mostovoy \cite{Mostovoy2, Mostovoy1} who realised that from several viewpoints the latter should be regarded as the ``right'' notion of lower central series for loops. In particular, the lower central series defined in terms of co-smash products as above does not coincide with the concept considered in~\cite{Huq}.
\end{example}

The binary commutator $[K,L]$ is also studied in~\cite{MM-NC}, where it is called the \defn{Higgins commutator}. It is an conceptual generalisation of the commutator which was introduced by Higgins in a varietal context~\cite{Higgins}. This definition should also be compared with the Huq commutator, as indeed in general, the two are different---but not too different.

\subsection{The Huq commutator}
By definition, a coterminal pair
\begin{equation*}\label{Cospan}
\vcenter{\xymatrix{K \ar[r]^-{k} & X & L \ar[l]_-{l}}}
\end{equation*}
of morphisms in a homological category \defn{Huq-commutes}~\cite{Bourn2002, BG, Huq} if and only if there is a (necessarily unique) morphism $\varphi$ such that the diagram
\[
\xymatrix@!0@=3em{ & K \ar[ld]_{\langle 1_{K},0\rangle} \ar[rd]^-{k} \\
K\times L \ar@{.>}[rr]|-{\varphi} && X\\
& L \ar[lu]^{\langle 0,1_{L}\rangle} \ar[ru]_-{l}}
\]
is commutative. We shall mainly be interested in the case where $k$ and $l$ are normal monomorphisms (= kernels). The \defn{Huq commutator
\[
[k,l]^{\Huq}\colon {[K,L]^{\Huq}\to X}
\]
of~$k$ and~$l$} is the smallest normal subobject of $X$ that should be divided out to make $k$ and $l$ commute---so that they do commute if and only if~${[K,L]^{\Huq}=0}$. This object may be obtained through the colimit $Q$ of the outer square above, as the kernel of the (regular epi)morphism ${X\to Q}$. In a homological category, an object~$X$ is abelian if and only if~$[X,X]^{\Huq}=0$.

\begin{remark}\label{Remark-Huq-Join}
In contrast with the Huq commutator, the Higgins commutator $[K,L]$ need not be normal in $X$, not even when both $K$ and $L$ are normal subobjects of $X$. In fact, the Huq commutator $[K,L]^{\Huq}$ of $K$, $L\normal X$ is the normal closure of~$[K,L]$, so that~${[[K,L],X] \join[K,L]=[K,L]^{\Huq}}$ by the following proposition, which is not explicitly needed further on:
\end{remark}

\begin{proposition}\cite{Actions}
If $K$, $L\leq X$ in a semi-abelian category then the normal closure of $K$ in the join $K \join L$ is $[K,L] \join K$.\noproof
\end{proposition}

\begin{remark}\label{abobj}
In a Mal'tsev category~$\A$, an object $A$ is said to be \defn{abelian} if and only if it carries a (necessarily unique) internal Mal'tsev operation: that is, a morphism $g\colon{A\times A\times A\to A}$ for which $g(x,x,z)=z$ and ${g(x,z,z)=x}$~\cite{Johnstone:Maltsev, Pedicchio}. As soon as $\A$ is moreover pointed, such an internal Mal'tsev operation is the same thing as an internal abelian group structure. However, in general, the two concepts are different. To avoid confusion, we denote the full subcategory of $\A$ determined by the abelian objects $\Mal(\A)$, and we write $\Ab(\A)$ for the category of internal abelian groups in $\A$.

For instance, an abelian object in the category of groups is an abelian group, and an abelian associative algebra over a field is a vector space (equipped with a trivial multiplication).

Note that an object $X$ in a finitely cocomplete homological category is abelian if and only if its commutator morphism $c_2^X$ is trivial: indeed, $[X,X]=0$ precisely when ${[X,X]^{\Huq}=0}$.
\end{remark}

\begin{remark}
The higher-order commutators are generally not built up out of iterated binary commutators (Remark~\ref{Remark-Associativity}, Example~\ref{Group-Commutators}, Example~\ref{Example-Loops}).
\end{remark}

The following basic properties will be useful throughout the text.

\begin{proposition}\cite{CCC}\label{Proposition-Commutator-Rules}
Let $X_{1}$, \dots, $X_{n}$ be subobjects of an object $X$ in $\A$.
\begin{enumerate}
\item[(o)] Commutators are reduced: if $X_{i}=0$ for some $i$ then $[X_{1},\dots,X_{n}]=0$. 
\item Commutators are symmetric: for any permutation ${\sigma\in \Sigma_{n}}$, 
\[
[X_{1},\dots,X_{n}]\cong [X_{\sigma^{-1}(1)},\dots,X_{\sigma^{-1}(n)}].
\]
\item Commutators are preserved by direct images: for $f\colon {X\to Y}$ regular epi,
\[
f[X_{1},\dots,X_{i},\dots, X_{n}]=[f(X_{1}),\dots,f(X_{i}),\dots, f(X_{n})].
\]
\item Commutators are monotone: if $M\leq X_{i}$ then
\[
[X_{1},\dots, X_{i-1},M,X_{i+1},\dots, X_{n}]\leq[X_{1},\dots, X_{i-1},X_{i},X_{i+1},\dots, X_{n}].
\]
\item Removing brackets enlarges the object:
\[
[[X_{1},\dots,X_{i}],X_{i+1},\dots, X_{n}]\leq [X_{1},\dots,X_{i},X_{i+1},\dots, X_{n}].
\]
\item Removing duplicates enlarges the object: if $X_{i}=X_{i+1}$ then
\[
[X_{1},\dots,X_{i},X_{i+1},X_{i+2},\dots, X_{n}]\leq [X_{1},\dots,X_{i},X_{i+2}\dots, X_{n}].
\]
\item When $\A$ is semi-abelian, if $X_1\vee\cdots \vee X_n = X$ then $[X_1,\ldots,X_n]\normal X$.
\end{enumerate}
\end{proposition}
\begin{proof}
(o) If $X_{i}=0$ then the morphism $\pi_{\coprod_{j\neq i}X_{j}}\comp r$ in Definition~\ref{co-smash product} becomes an isomorphism. Hence $X_{1}\tensor \cdots\tensor X_{n}$ is zero. For (i) it suffices to note that the definition of $X_{1}\tensor \cdots\tensor X_{n}$ is symmetric in the objects $X_{1}$, \dots, $X_{n}$. Statement~(ii) is a consequence of Proposition~\ref{cr-reg-epi}. (iii) holds because the inclusion of $M$ into $X_{i}$ gives a morphism
\[
X_{1}\tensor\cdots\tensor X_{i-1}\tensor M\tensor X_{i+1}\tensor\cdots\tensor X_{n}\to X_{1}\tensor\cdots\tensor X_{i-1}\tensor X_{i}\tensor X_{i+1}\tensor \cdots\tensor X_{n}
\]
which induces the needed inclusion of commutators. The inclusion in (iv) is induced by the bracketing comparison
\[
(X_{1}\tensor\cdots\tensor X_{i})\tensor X_{i+1}\tensor\cdots\tensor X_{n}\to X_{1}\tensor\cdots\tensor X_{i}\tensor X_{i+1}\tensor \cdots\tensor X_{n}
\]
from Remark~\ref{Remark-Associativity} and the fact that tensor products preserve regular epimorphisms, Proposition~\ref{cr-reg-epi}. The inclusion in (v) is obtained through (possibly higher-order versions of) the folding operations introduced in Notation~\ref{Folding Operations} below. For the proof of (vi) it suffices to recall that in a semi-abelian category, the direct image of a kernel along a regular epimorphism is still a kernel.
\end{proof}

\begin{proposition}\label{Proposition-Join}
Commutators satisfy a distribution rule with respect to joins:
\[
[X_1,\ldots,X_n,A_1\vee\cdots \vee A_m] = \bigvee_{\stackrel{\scriptstyle 1\leq k\leq m}{1\leq i_1<\ldots<i_k\leq m}}
[X_1,\ldots,X_n,A_{i_1},\ldots , A_{i_m}].
\]
\end{proposition}
\begin{proof}
We show that for all $X$ in $\A$ and $A$, $B\leq X$ the equality
\[
[X,A\vee B]=[X,A,B]\vee [X,A]\vee [X,B]
\]
holds, and refer to~\cite{CCC} for the general case, of which the proof is similar. Lemma~\ref{Tensor vs sum} tells us that 
\[
X\tensor(A+B)=(X\tensor A\tensor B)\rtimes (X\tensor A)\bigr)\rtimes (X\tensor B).
\]
Using the equality $A\vee B=\Im(\som{a}{b}\colon A+B\to X)$ and the fact that $X\tensor(-)$ preserves regular epimorphisms (Proposition~\ref{cr-reg-epi}), we may embed the co-smash products into the respective sums and take images of the composites with the inclusions $1_{X}$, $a$ and~$b$ into $X$ to obtain the needed join decomposition of the commutator.
\end{proof}

\begin{notation}\label{Folding Operations}
Given objects $X$ and $Y$ in $\A$, we consider the \defn{folding operations}
\[
S_{1,2}^{X,Y}\colon {X\tensor Y\tensor Y\to X\tensor Y}\qquad\text{and}\qquad S_{2,1}^{X,Y}\colon {X\tensor X\tensor Y\to X\tensor Y}.
\]
The one on the left is induced by
\[
\xymatrix@=3em{X\tensor Y\tensor Y \ar[r]^-{\iota_{X,Y,Y}} & X+Y+Y \ar[r]^-{1_{X}+\nabla_{Y}} & X+Y \ar[r]^-{\left\langle\begin{smallmatrix} 1_{X} & 0 \\
0 & 1_{Y}
\end{smallmatrix}\right\rangle} & X\times Y}
\]
being trivial, as indeed the diagram
\[
\xymatrix{X+Y+Y \ar[d]_-{1_{X}+\nabla_{Y}} \ar[r]^-{r} & (X+Y)^{2}\times (Y+Y) \ar[d]^-{(\rho_{X}\circ\pi_{1})\times \nabla_{Y}}\\
X+Y \ar[r]_-{\left\langle\begin{smallmatrix} 1_{X} & 0 \\
0 & 1_{Y}
\end{smallmatrix}\right\rangle} & X\times Y}
\]
commutes, and the one on the right is given by the analogous argument. 
\end{notation}

\begin{proposition}\label{Split-Exact}
Suppose that $\A$ is finitely cocomplete homological. Let $X$ be an object of $\A$. Then any split right-exact sequence 
\[
\xymatrix{Y \ar[r]^-{\del } & V \ar@{-{ >>}}@<.5ex>[r]^-{p} & Z \ar@{{ >}->}@<.5ex>[l]^-s\ar[r] & 0}
\]
gives rise to a split exact sequence
\begin{equation}\label{Sequence-Rightex-Split-cross}
{\xymatrix@=3em{(X\tensor Y\tensor Z) \rtimes (X\tensor Y) \ar[rr]^-{\bigsom{S_{1,2}^{X,V}\circ (1_{X}\tensor \del\tensor s)}{1_{X}\tensor \del }} && X\tensor V \ar@{-{ >>}}@<.5ex>[r]^-{1_{X}\tensor p} & X\tensor Z \ar[r] \ar@{{ >}->}@<.5ex>[l]^-{1_X\tensor s}& 0.}}
\end{equation}
\end{proposition}
\begin{proof}
Consider the diagram of solid arrows
\[
\xymatrix@=3em{
&(X\tensor Y\tensor Z) \rtimes (X\tensor Y) \ar[r]^-{\bigsom{\iota^{X\tensor (-)}_{Y,Z}}{1_{X}\tensor \iota_Y}} \ar@{.{ >>}}[d] & X\tensor (Y+Z) \ar@{-{ >>}}[r]^-{1_{X}\tensor \rho_Z} \ar@{-{ >>}}[d]_-{1_{X}\tensor \bigsom{\del}{s}} & X\tensor Z \ar@{=}[d]\ar[r] & 0\\
0 \ar[r] & \Ker(1_{X}\tensor p) \ar@{{ |>}->}[r]& X\tensor V \ar@{-{ >>}}[r]_-{1_{X}\tensor p} & X\tensor Z \ar[r] &0.}
\]
Its top row is exact by Lemma~\ref{Tensor vs sum}; in particular, the top left morphism is proper: in the notation of Lemma~\ref{Tensor vs sum}, its image is $W$. Moreover, $1_{X}\tensor \bigsom{\del}{s}$ is a regular epimorphism by the protomodularity of $\A$ and by Proposition~\ref{cr-reg-epi}, which says that~$X\tensor (-)$ preserves regular epimorphisms. By the uniqueness of image factorisations, the kernel of $1_{X}\tensor p$ is equal to
\begin{align*}
\im\Bigl((1_{X}\tensor \bigsom{\del}{s}) \comp \bigsom{\iota^{X\tensor (-)}_{Y,Z}}{1_{X}\tensor \iota_Y}\Bigr)
&= \im\Bigl(\bigsom{S_{1,2}^{X,V}\circ (1_{X}\tensor \del\tensor s)}{1_{X}\tensor \del}\Bigr).
\end{align*}
Hence the sequence~\eqref{Sequence-Rightex-Split-cross} is exact.
\end{proof}

\section{Semi-direct products}\label{Section-Semi-Direct-Products}

Internal actions were first introduced in~\cite{Bourn-Janelidze:Semidirect}, then studied in detail in~\cite{BJK}. As explained in the introduction, following~\cite{Actions}, we shall manipulate internal actions through co-smash products. More precisely, given a split epimorphism $p$ with splitting $s$ corresponding to an algebra $\xi\colon{B\flat X\to X}$, we study the properties of the point $(p,s)$ in terms of the restriction $\psi\colon{B\tensor X\to X}$ of $\xi$ to $B\tensor X$ rather than via $\xi$ itself. The present section sketches how this works; for more details we refer to~\cite{Actions}.

\subsection{Basic analysis}\label{basic analysis}
Suppose that $A$ and $G$ are objects of $\A$ and $\psi\colon A\tensor G\to A$ is a morphism for which the induced morphism $k_{\psi} =q \comp \iota_A$ in the diagram
\[
\xymatrix{& A \ar@{.>}[rd]^-{\iota_{A}} \ar@/^/@{.>}[rrd]^-{k_{\psi}} \\
A\tensor G \ar@{.>}[ru]^-{\psi} \ar@<-.5ex>[rr]_-{\iota_{A,G}} \ar@<.5ex>[rr]^-{\iota_A\circ \psi} && A+G \ar@{-{ >>}}[r]^-q & \Coeq(\iota_{A,G},\iota_A\comp \psi)=Q}
\]
is a monomorphism. We may then write $A\rtimes_{\psi}G=Q$ and call the object $Q$ the \defn{semi-direct product of $A$ and $G$ along~$\psi$}, as it fits into the diagram
\begin{equation*}
\xymatrix@=3em{0 \ar[r] & A\tensor G \ar@{{ |>}->}[r]^-{\iota_{A,G}} \ar[d]_-{\psi} & A+G \ar@{-{ >>}}[r]^-{\langle \rho_{A},\rho_{G}\rangle} \ar[d]_-{q} \ar@<.5ex>@{.>}[rd]^(.4){\rho_{G}} & A\times G \ar[d]^-{\pi_{G}} \ar[r] & 0\\
0 \ar[r] & A \ar@{{ |>}->}[r]_-{k_{\psi}} & A\rtimes_{\psi}G \ar@{-{ >>}}@<.5ex>[r]^-{p_{\psi}} & G \ar@{.>}@<.5ex>[ul]^(.6){\iota_{G}} \ar[r] \ar@{{ >}->}@<.5ex>[l]^-{s_{\psi}} & 0.}
\end{equation*}
The rows are short exact sequences, $p_{\psi}$ is induced by $\rho_G\colon A+G\to G$ and $s_{\psi} =q\comp \iota_G$.

That is to say, if $k_{\psi}$ is a monomorphism then the morphism $\psi\colon A\tensor G\to A$ gives rise to a point $(p_{\psi},s_{\psi})\colon{A\rtimes_{\psi}G\point G}$. It is not difficult to see that conversely, any point $(p,s)\colon{X\point G}$ will give rise to such a $\psi$ through the diagram with short exact rows
\[
\xymatrix@C=4em{0 \ar[r] & A\tensor G \ar@{{ |>}->}[r]^-{\iota_{A,G}} \ar@{.>}[d]_-{\psi} & A+G \ar@{-{ >>}}[r]^-{\langle \rho_{A},\rho_{G}\rangle} \ar[d]_-{\bigsom{a}{s}} & A\times G \ar[d]^-{\pi_{G}} \ar[r] & 0\\
0 \ar[r] & A \ar@{{ |>}->}[r]_-{a} & X \ar@{-{ >>}}@<.5ex>[r]^-{p} & G \ar@{{ >}->}@<.5ex>[l]^-{s} \ar[r] & 0,}
\]
and that the two constructions are each other's inverse. 

Hence a morphism $\psi\colon A\tensor G\to A$ is induced by an action if and only if the above-mentioned condition holds. (Compare with the analysis of actions worked out in~\cite{MFS}.) We shall then say that $\psi$ \defn{induces} or \defn{determines} an action and sometimes, abusively, that $\psi$ \defn{is} an action. Note that here we consider $\psi$ as a morphism defined on $A\tensor G$ instead of $G\tensor A$ as in the introduction.

\begin{example}\label{Action-G-group}
In the category of groups a morphism $\psi\colon{A\tensor G \to A}$ induces an action if and only if the function
\[
(g,a)\mapsto g\cdot a =\psi(gag^{-1}a^{-1})a
\]
is a $G$-group structure, that is, it does not only satisfy the rules $1\cdot a =a$ and $(gg')\cdot a=g\cdot (g'\cdot a)$, but also $g\cdot (a a')=(g\cdot a)(g\cdot a')$. This agrees with the fact that in $\Gp$, semi-direct products correspond with $G$-groups rather than with general actions; see the detailed discussion in~\cite{Actions}.
\end{example}

\begin{example}\label{Example-Conjugation}
For any object $X$, the morphism
\[
c^{X,X}=c^{X}_{2}=\nabla_{X}^{2}\comp\iota_{X,X}\colon{X\tensor X\to X}
\]
induces an action of $X$ on itself called the \defn{conjugation action}~\cite{Janelidze-Marki-Ursini}, which corresponds to the split short exact sequence
\[
\xymatrix@=3em{0 \ar[r] & X \ar@{{ |>}->}[r]^-{\langle1_{X},0\rangle} & X\times X \ar@{-{ >>}}@<.5ex>[r]^-{\pi_{2}} & X \ar[r] \ar@{{ >}->}@<.5ex>[l]^-{\langle1_{X},1_{X}\rangle} & 0.}
\]
\end{example}

\begin{proposition}\cite{Actions}\label{conjnormsub}
Let $n\colon N \to X$ be a normal monomorphism. Then the unique morphism $c^{N,X}\colon N\tensor X \to N$ for which the diagram
\[
\xymatrix{N\tensor X \ar@{.>}[r]^-{c^{N,X}} \ar[d]_-{n\tensor 1_{X}} & N \ar@{{ |>}->}[d]^-{n} \\
X\tensor X \ar[r]_-{c^{X,X}} & X}
\]
commutes determines an action, called the \defn{conjugation action of $X$ on $N$}. This process is natural in the sense that any commutative diagram as on the left
\[
\xymatrix{N \ar[r] \ar@{{ |>}->}[d] & N' \ar@{{ |>}->}[d]\\
X \ar[r] & X'}
\qquad\qquad
\xymatrix{N\tensor X \ar[r] \ar[d]_-{c^{N,X}} & N'\tensor X' \ar[d]^-{c^{N',X'}}\\
N \ar[r] & N'}
\]
gives a commutative diagram as on the right.
\end{proposition}
\begin{proof}
Writing $q\colon {X\to Q}$ for the cokernel of $n$, we see that
\[
q\comp c^{X,X}\comp (n\tensor 1_{X})=c^{Q,Q}\comp (q\tensor q)\comp (n\tensor 1_{X})=c^{Q,Q}\comp (0\tensor q)=0,
\]
so $c^{X,X}\comp (n\tensor 1_{X})$ factors over $n$. This $c^{N,X}$ satisfies the condition of~\ref{basic analysis}.
\end{proof}

\begin{proposition}[Co-universal property of the semi-direct product, \cite{Actions}]\label{PU}
Consider an action induced by $\psi\colon A\tensor G \to A$ and morphisms
\[
\xymatrix{A \ar[r]^-{f} & Z & G. \ar[l]_-{g}}
\]
Then there exists a (necessarily unique) morphism $\som{f}{g}\colon A \rtimes_{\psi} G \to Z$ such that $\som{f}{g}\comp k_{\psi} =f$ and $\som{f}{g}\comp s_{\psi} =g$ if and only if the diagram
\[\label{PUdia}
\vcenter{\xymatrix{A\tensor G \ar[r]^-{\psi} \ar[d]_-{f\tensor g} & A \ar[d]^-f\\
Z\tensor Z \ar[r]_-{c_2^Z} & Z}}
\]
commutes.
\end{proposition}
\begin{proof}
The morphism $\som{f}{g}\colon{A+G\to Z}$ factors over $q\colon{A+G\to A \rtimes_{\psi} G}$ because
\begin{align*}
\som{f}{g}\comp \iota_{A}\comp\psi &= f\comp \psi = c_{2}^{Z}\comp (f\tensor g) = \nabla_{Z}\comp \iota_{Z,Z}\comp (f\tensor g)\\
& = \nabla_{Z}\comp (f+ g)\comp\iota_{A,G}=\som{f}{g}\comp \iota_{A,G}.
\end{align*}
This factorisation is clearly unique. Now $\som{f}{g}\comp s_{\psi}= \som{f}{g}\comp q\comp \iota_{G}=g$ and $\som{f}{g}\comp k_{\psi}=\som{f}{g}\comp q\comp \iota_{A}=f$, which finishes the ``if''-part of the proof.
\end{proof}

\begin{example}\label{Trivial-Action}
The \defn{trivial action} of an object $G$ on an object $A$ is induced by the zero morphism $0\colon {A\tensor G\to A}$. Here the semi-direct product ${A\rtimes_{0} G}$ is $A\times G$ and~$p_{0}$ is the product projection $\pi_{G}\colon {A\times G\to G}$. Hence two coterminal morphisms~$f$ and $g$ as in Proposition~\ref{PU} Huq-commute if and only if $c_{2}^{Z}\comp (f\tensor g)$ is trivial. This of course also follows immediately from the fact that $A\times G$ is the cokernel of~$\iota_{A,G}\colon{A\tensor G\to A+G}$ and the equality $c_{2}^{Z}\comp (f\tensor g)=\som{f}{g}\comp \iota_{A,G}$.
\end{example}

\begin{example}[Centrality]\label{Example-Centrality}
The conjugation action $c^{N,X}$ of an object $X$ on a normal subobject $N\normal X$ is trivial if and only if $N$ is \defn{central} in $X$, which means that $n\colon{N\to X}$ and $1_{X}\colon{X\to X}$ Huq-commute~\cite{Bourn2002}; indeed $n\comp c^{N,X}=c^{X,X}\comp (n\tensor 1_{X})$. (Compare with Theorem~3.2.4 in~\cite{AlanThesis}.)
\end{example}

Starting from conjugation actions we may again construct various new actions by the following device, of which the proof is immediate from~\ref{basic analysis}.

\begin{proposition}\cite{Actions}\label{stabact}
Suppose that $\psi\colon A\tensor G \to A$ induces an action, $m\colon {M\to A}$ is a monomorphism and $h\colon H\to G$ a morphism. Suppose that $M$ is~\defn{$H$-stable under~$\psi$}, that is, $\psi\comp (m\tensor h)\colon M\tensor H \to A$ factors through a (necessarily unique) $\varphi\colon M\tensor H \to M$ such that $\psi\comp (m\tensor h)=m\comp\varphi$. Then $\varphi$ induces an action of $H$ on~$M$.\noproof
\end{proposition}

\begin{notation}\label{Notation-Pulation}
If $M=A$ in the above proposition then we write $ \varphi = h^*(\psi)$
\begin{equation*}\label{Pullback}
\vcenter{\xymatrix@!0@R=3.5em@C=5em{A\tensor H \ar[r]^-{h^{*}(\psi)} \ar[d]_-{1_{A}\tensor h} & A \ar@{=}[d] \\
A\tensor G \ar[r]_-{\psi} & A}}
\end{equation*}
and call $\varphi$ the \defn{pullback} of $\psi$ along $h$. This choice of terminology is explained by the fact that the above diagram matches the morphism of split short exact sequences
\[
\xymatrix@!0@R=3.5em@C=5em{0 \ar[r] & A \ar@{=}[d] \ar@{{ |>}->}[r]^-{k_{\varphi}} & A\rtimes_{\varphi} H \ar@<.5ex>@{-{ >>}}[r]^-{p_{\varphi}} \ar[d]_-{1_{A}\rtimes h} & H \ar[d]^-{h} \ar[r] \ar@<.5ex>@{{ >}->}[l]^-{s_{\varphi}} & 0\\
0 \ar[r] & A \ar@{{ |>}->}[r]_-{k_{\psi}} & A\rtimes_{\psi} G \ar@<.5ex>@{-{ >>}}[r]^-{p_{\psi}} & G \ar[r] \ar@<.5ex>@{{ >}->}[l]^-{s_{\psi}} & 0.}
\]
It is well known that now the right hand side square of the diagram is a pullback~\cite{Bourn2001}. In fact, one easily sees that it is also a pushout.
\end{notation}

\begin{example}\label{Example-Normal}
If $N\normal X$ as in Proposition~\ref{conjnormsub} then
\[
n^{*}(c^{N,X})=c^{N,N}=c^{N}_{2}.
\]
Indeed, $n\comp c^{N,X}\comp(1_{N}\tensor n)=c^{X,X}\comp (n\tensor 1_{X})\comp (1_{N}\tensor n)=c^{X}_{2}\comp (n\tensor n)$, which equals $n\comp c^{N}_{2}$ by naturality of conjugation actions.
\end{example}

\begin{example}\label{Example-m=conj}
If ${\psi}\colon A\tensor G \to A$ determines an action then
\[
\psi = c^{A,A\rtimes_{\psi} G} \comp (1_{A}\tensor s_{\psi}) = s_{\psi}^{*}(c^{A,A\rtimes_{\psi} G}).
\]
This means that the action determined by $\psi$ coincides with the restriction to $G$ of the conjugation action of the semi-direct product ${A\rtimes_{\psi} G}$ on $A$.
\end{example}

\subsection{The induced higher-order operations}
Internal actions induce certain higher-order operations defined as follows.

\begin{notation}\label{Notation-Higher-Orders}
Let $A$ and $G$ be objects and $\psi\colon A\tensor G \to A$ a morphism. Consider $n\geq 2$ and $1\leq k\leq n-1$. Define $\psi_{k,n-k} $ to be the composite morphism
\[
\psi_{k,n-k}\colon \xymatrix@=4em{A\tensor \cdots\tensor A\tensor G\tensor \cdots \tensor G \ar[r]^-{S_{k,n-k}^{A,G}} & A\tensor G \ar[r]^-{\psi} & A.}
\]
In particular, if we take $\psi=c^{N,X}$ to be induced by the conjugation action of an object~$X$ on some normal subobject $N\normal X$, then we get morphisms
\[
c^{N,X}_{k,n-k}\colon N\tensor \cdots\tensor N\tensor X\tensor \cdots \tensor X\to N.
\]
\end{notation}

Note that $c^{N,X}_{1,1}=c^{N,X}$. Also the higher-order operations $c^{N,X}_{k,n-k}$ are interrelated, the generic relation being the following one: 

\begin{lemma}\label{c21=c12} 
For any normal monomorphism $n\colon N\to X$ the equality
\[
c_{2,1}^{N,X} = c_{1,2}^{N,X}\comp (1_{N}\tensor n\tensor 1_{X})\colon N\tensor N\tensor X\to N
\]
holds. In particular, $c_{2,1}^{X,X} = c_{1,2}^{X,X} = c_3^X$.
\end{lemma}
\begin{proof}
Post-compose with $n$ and use the commutative diagrams obtained by injecting the various co-smash products into the corresponding sums.
\end{proof}

This coherence condition in terms of ternary co-smash products will appear again in the analysis of crossed modules: see for instance Theorem~\ref{Theorem-Characterisation-XMod} below. We shall also investigate some closely related structures such as Beck modules, which satisfy variations of this condition (see Section~\ref{Section-Beck}).

\section{The \emph{Smith is Huq} condition}\label{Section-SH}
We explain how the \emph{Smith is Huq} condition for finitely cocomplete homological categories may be expressed in terms of co-smash products as the vanishing of a ternary commutator. Thus a condition which is about \emph{locally defined internal categorical structures admitting a global extension} is characterised as a computational obstruction. This is \emph{the} key point of the present article---all results in the ensuing sections are based on it.

Theorem~\ref{Theorem-SH-1} characterises when two given equivalence relations $R$, $S$ on a common object $X$ commute in the Smith sense: if $K$ and $L$, respectively, denote their denormalisations, then
\[
[K,L]=0=[K,L,X]
\] 
is a necessary and sufficient condition. This immediately gives a characterisation of the \emph{Smith is Huq} condition (Theorem~\ref{Theorem-SH}) and a formula for the Smith commutator in terms of co-smash products (Theorem~\ref{Theorem-Smith-Commutator}). We also find a characterisation of double central extensions (Proposition~\ref{Characterisation-Double-Central-Extensions}), which allows us to make the Hopf formula for the third homology of an object in any semi-abelian category with enough projectives explicit (Theorem~\ref{Theorem-H3}).

\subsection{The Smith commutator}
Consider a pair of equivalence relations $(R,S)$ on a common object $X$
\begin{equation*}\label{Category-RG}
\xymatrix@!0@=5em{R \ar@<1.5ex>[r]^-{r_{1}} \ar@<-1.5ex>[r]_-{r_{2}} & X \ar[l]|-{\Delta_{R}} \ar[r]|-{\Delta_{S}} & S, \ar@<1.5ex>[l]^-{s_{1}} \ar@<-1.5ex>[l]_-{s_{2}}}
\end{equation*}
and consider the induced pullback of $r_{1}$ and $s_{2}$.
\begin{equation*}\label{Pullback-RS}
\vcenter{\xymatrix@!0@=4em{R\times_{X}S \ar@{}[rd]|<<{\pullback} \ar[r]^-{\pi_{S}} \ar[d]_-{\pi_{R}} & S \ar[d]^-{s_{2}} \\
R \ar[r]_-{r_{1}} & X}}
\end{equation*}
The equivalence relations $R$ and $S$ are said to \defn{Smith-commute}~\cite{Smith,Pedicchio,BG} if and only if there is a (necessarily unique) morphism $\theta$ (a \defn{connector} between $R$ and~$S$) for which the diagram
\[
\xymatrix@!0@=3em{ & R \ar[ld]_{\langle 1_{R},\Delta_{S} \circ r_{1}\rangle} \ar[rd]^-{r_{2}} \\
R\times_{X}S \ar@{.>}[rr]|-{\theta} && X\\
& S \ar[lu]^{\langle\Delta_{R} \circ s_{2},1_{S}\rangle} \ar[ru]_-{s_{1}}}
\]
is commutative. The connector $\theta$ is a partially defined Mal'tsev operation on~$X$, as the diagram commutes precisely when $\theta(x,x,z)=z$ for $(x,z)\in S$ and $\theta(x,z,z)=x$ for $(x,z)\in R$. It is also the same thing as a \defn{pregroupoid structure}~\cite{Kock-Pregroupoids, Johnstone:Herds} on the span $(d=\coeq(r_{1},r_{2}), c=\coeq(s_{1},s_{2}))$.

The \defn{Smith commutator $[R,S]^{\Smith}$ of~$R$ and~$S$} is the smallest equivalence relation on $X$ that should be divided out to make $R$ and $S$ commute, so that they do commute if and only if~${[R,S]^{\Smith}=\Delta_{X}}$. This equivalence relation may be obtained through the colimit $Q$ of the outer square above, as the kernel pair of the (regular epi)morph\-ism~${X\to Q}$.

\subsection{The \emph{Smith is Huq} condition}\label{Subsection-SH}
The \defn{normalisation} $K$ of an equivalence relation $(R,r_{1},r_{2})$ on $X$ is the monomorphism
\[
r_{2}\comp\ker(r_{1})\colon K=\Ker(r_{1})\to X.
\]
A monomorphism is called an \defn{ideal} if and only if it is the normalisation of some (necessarily unique) equivalence relation~\cite{Bourn2000}. In a homological category, ideals are direct images of kernels along regular epimorphisms---see~\cite{MM-NC} for an in-depth analysis. For now, it suffices to note that the normalisation of an \emph{effective} equivalence relation is always a kernel; conversely, any normal subobject~$N\normal X$ (in the strong sense that it may be represented by a kernel) admits a \defn{denormalisation}~$R_{N}$, the kernel pair of its cokernel. This process determines an order isomorphism between the normal subobjects of $X$ and the effective equivalence relations on~$X$, which in the semi-abelian case coincides with the correspondence between ideals and equivalence relations.

It is well known that Smith-commuting equivalence relations always have Huq-commuting normalisations~\cite{BG}. However, the converse need not hold: Janelidze gave a counterexample in the category of digroups~\cite{Borceux-Bourn, Bourn2004}, which is a semi-abelian variety, even a variety of~$\Omega$-groups~\cite{Higgins}. (See also Example~\ref{Example-Loops}.) Thus arises a property homological categories may or may not have: 

\begin{definition}\label{Definition-SH}
A homological category satisfies the \defn{Smith is Huq} condition \defn{(SH)} if and only if two effective equivalence relations on a given object always commute as soon as their normalisations do.
\end{definition}

It turns out that the condition (SH) is fundamental in the study of internal categorical structures: it is shown in~\cite{MFVdL} that, for a semi-abelian category, this condition holds if and only if every star-multiplicative graph is an internal groupoid. As explained in~\cite{Janelidze} and in Section~\ref{Section-Crossed-Modules} of the present article, this is important when characterising internal crossed modules. 

The \emph{Smith is Huq} condition is known to hold for pointed strongly protomodular exact categories~\cite{BG} (in particular, for any Moore category~\cite{Rodelo:Moore}) and for action accessible categories~\cite{BJ07, AlanThesis} (in particular, for any category of interest~\cite{Montoli, Orzech}). Well-known examples are the categories of groups, Lie algebras, associative algebras, non-unitary rings, and (pre)crossed modules of groups.

\begin{theorem}\label{Theorem-SH-1}
In a finitely cocomplete homological category, consider effective equivalence relations $R$ and $S$ on $X$ with normalisations $K$, $L\normal X$, respectively. Then the following are equivalent:
\begin{enumerate}
\item $R$ and $S$ Smith-commute;
\item $[K,L]=0=[K,L,X]$.\noproof
\end{enumerate}
\end{theorem}

Hence a homological category satisfies (SH) if and only if for every pair of effective equivalence relations of which the normalisations commute, \emph{the ternary commutator obstruction vanishes}. The proof is an obvious application of the following fundamental lemma (take $\beta=1$). The basic admissibility condition which appears in it was first discovered by Martins--Ferreira~\cite{MF-PhD,NMF2}. (Incidentally, we believe that Lemma~\ref{Lemma-Nelson} answers part of the question asked in the concluding section of that paper; see also~\cite{MFVdL3}.) We shall consider diagrams of shape
\begin{equation}\label{couniv}
\vcenter{\xymatrix@!0@=4em{A \ar@<.5ex>[r]^-{f} \ar[rd]_-{\alpha} & B
\ar@<.5ex>[l]^-{r}
\ar@<-.5ex>[r]_-{s}
\ar[d]^-{\beta} & C \ar@<-.5ex>[l]_-{g} \ar[ld]^-{\gamma}\\
& D}}
\end{equation}
with $f\comp r=1_{B}=g\comp s$ and $\alpha \comp r=\beta=\gamma \comp s$. By taking the pullback of $f$ with $g$, any diagram such as~\eqref{couniv} may be extended to a diagram
\[
\vcenter{\xymatrix@!0@=3em{ & C \ar@<.5ex>[ld]^-{e_2} \ar@<-.5ex>[rd]_-{g}
\ar@/^/[rrrd]^-{\gamma} \\
A\times_{B}C \ar@<.5ex>[ru]^-{\pi_C}
\ar@<-.5ex>[rd]_-{\pi_A} && B \ar@<.5ex>[ld]^-{r} \ar@<-.5ex>[lu]_-{s}
 \ar[rr]|-{\beta} && D\\
& A \ar@<.5ex>[ru]^-{f} \ar@<-.5ex>[lu]_-{e_1} \ar@/_/[urrr]_-{\alpha}}}
\]
in which the square is a double split epimorphism (that is, also the obvious squares involving splittings commute). The triple $(\alpha,\beta,\gamma)$ is said to be \defn{admissible with respect to $(f,r,g,s)$} if and only if there exists a (necessarily unique) morphism $\vartheta\colon{A\times_{B}C\to D}$ such that $\vartheta \comp e_{1}=\alpha$ and $\vartheta \comp e_{2}=\gamma$.

\begin{lemma}\label{Lemma-Nelson}
Given any diagram~\eqref{couniv}, let $\overline{k}\colon{\overline{K}\to D}$ be the image of $\alpha\comp \ker (f)$, $\overline{l}\colon{\overline{L}\to D}$ the image of $\gamma\comp \ker (g)$ and $\overline{\beta}\colon{\overline{B}\to D}$ the image of $\beta$. Then the triple $(\alpha,\beta,\gamma)$ is admissible with respect to $(f,r,g,s)$ if and only if
\[
[\overline{K},\overline{L}]=0=[\overline{K},\overline{L},\overline{B}].
\] 
\end{lemma}
\begin{proof}
We decompose $A$, $C$ and $A\times_{B}C$ into semi-direct products and then analyse in terms of the induced actions what it means for $\vartheta$ to exist. There are unique $\varphi$ and $\psi$ that give rise to the morphisms of split short exact sequences
\[
\vcenter{\xymatrix@!0@C=5em@R=4em{
0\ar[r] & K\ar@{=}[d] \ar@{{ |>}->}[r]^-{{\ker}(f)}& A 
\ar@<.5ex>@{-{ >>}}[r]^-{f} 
&B
\ar[r] \ar@<.5ex>@{{ >}->}[l]^-{r}
&0\\
0\ar[r] & K \ar@{{ |>}->}[r]_-{k_{\varphi}}& K\rtimes_{\varphi} B
\ar[u]^{\rho}_{\cong}
\ar@<.5ex>@{-{ >>}}[r]^-{p_{\varphi}} 
&B\ar@{=}[u]
\ar[r] \ar@<.5ex>@{{ >}->}[l]^-{s_{\varphi}}
&0
}}
\]
and
\[
\vcenter{\xymatrix@!0@C=5em@R=4em{
0\ar[r] & L\ar@{=}[d] \ar@{{ |>}->}[r]^-{{\ker}(g)}& C 
\ar@<.5ex>@{-{ >>}}[r]^-{g} 
&B
\ar[r] \ar@<.5ex>@{{ >}->}[l]^-{s}
&0\\
0\ar[r] & L \ar@{{ |>}->}[r]_-{k_{\psi}}& L\rtimes_{\psi} B
\ar[u]^{\sigma}_{\cong}
\ar@<.5ex>@{-{ >>}}[r]^-{p_{\psi}} 
&B \ar@{=}[u]
\ar[r] \ar@<.5ex>@{{ >}->}[l]^-{s_{\psi}}
&0.
}}
\]
By Notation~\ref{Notation-Pulation} we obtain the commutative diagram with exact rows
\[
\vcenter{\xymatrix@!0@C=8em@R=4em{
0\ar[r] & K\ar@{=}[d] \ar@{{ |>}->}[r]^-{k_{\zeta}}& K\rtimes_{\zeta} (L\rtimes_{\psi} B) \ar[d]^{\kappa}_{\cong}
\ar@<.5ex>@{-{ >>}}[r]^-{p_{\zeta}} 
&L\rtimes_{\psi} B
\ar[r] \ar@<.5ex>@{{ >}->}[l]^-{s_{\zeta}} \ar[d]^{\sigma}_{\cong}
&0\\
0\ar[r] & K\ar@{=}[d] \ar@{{ |>}->}[r]^-{\langle {\ker}(f),0\rangle}
& A\times_B C
\ar[d]^-{\pi_A}
\ar@<.5ex>@{-{ >>}}[r]^-{\pi_C} 
&C\ar[d]^-{g}
\ar[r] \ar@<.5ex>@{{ >}->}[l]^-{e_2=\langle r \circ g,1_C\rangle}
&0\\
0\ar[r] & K \ar@{{ |>}->}[r]_-{{\ker}(f)}& A 
\ar@<.5ex>@{-{ >>}}[r]^-{f} 
&B
\ar[r] \ar@<.5ex>@{{ >}->}[l]^-{r}
&0
}}
\]
in which $\zeta=(g \comp \sigma)^*(\varphi)=\varphi \comp (1_K\tensor p_{\psi})$. Now write $k=\alpha \comp {\ker}(f)\colon K\to D$ and $l=\gamma \comp {\ker}(g)\colon L\to D$. If the desired morphism $\vartheta$ exists then
\begin{align*}
\vartheta \comp \kappa &= \vartheta\comp\som{\langle{\ker}(f),0 \rangle}{
e_2 \comp \sigma}
= \vartheta\comp\som{\langle1_A,s\circ f\rangle \circ {\ker}(f)}
{e_2 \circ \sigma}
=\som{\vartheta \circ e_1 \circ {\ker}(f)}{\vartheta \circ e_2 \circ \sigma}\\
&=\som{\alpha \circ {\ker}(f)}{\gamma \circ \sigma}
=\bigsom{\alpha \circ {\ker}(f)}{
\som{\gamma \circ {\ker}(g)}{\beta}}
=\bigsom{k}{\som{l}{\beta}}.
\end{align*}
Conversely, if the morphism
\[
\vartheta'=\bigsom{k}{\som{l}{\beta}}
\]
exists then $\vartheta = \vartheta' \comp \kappa^{-1}$ satisfies the relevant constraints: it is clear from the above calculation that $
\vartheta' \comp \kappa^{-1} \comp e_2 =\gamma$ and that 
$
\vartheta' \comp \kappa^{-1} \comp e_1 \comp {\ker}(f)=\alpha \comp {\ker}(f)$,
but we also have
\begin{align*}
\vartheta' \comp \kappa^{-1} \comp e_1 \comp r &= \vartheta' \comp \kappa^{-1} \comp \langle 1_C,s\comp f\rangle \comp r = 
\vartheta' \comp \kappa^{-1} \comp \langle r,s\rangle
=\vartheta' \comp \kappa^{-1} \comp \langle r \comp g,1_C\rangle\comp s \\
&=
\vartheta' \comp \kappa^{-1} \comp e_2 \comp s=\gamma\comp s =\beta=\alpha \comp r.
\end{align*}
Thus $\vartheta' \comp \kappa^{-1} \comp e_1 =\alpha$. It follows that the desired morphism $\vartheta$ exists if and only if $\vartheta'$ exists, which according to Proposition~\ref{PU} is the case if and only if the diagram
\begin{equation}\label{fundlem-PUsquare}
\vcenter{\xymatrix{K\tensor (L\rtimes_{\psi}B) \ar[r]^-{\zeta} \ar[d]_-{k\tensor \som{l}{\beta}} & K \ar[d]^-{k} \\
D\tensor D \ar[r]_-{c^{D,D}} & D}}
\end{equation}
commutes. 
To find conditions for this to happen we use sequence~\eqref{Sequence-Rightex-Split-cross} from Proposition~\ref{Split-Exact} in order to decompose the object~$K\tensor (L\rtimes_{\psi}B)$ in three parts, via the regular epimorphism
\[
\somm{S_{1,2}^{K,L\rtimes_{\psi}B} \circ (1_K\tensor k_{\psi}\tensor s_{\psi})
}{1_{K}\tensor k_{\psi}}{1_{K}\tensor s_{\psi}}\colon(K\tensor L\tensor B)+(K\tensor L)+(K\tensor B)\to K\tensor (L\rtimes_{\psi}B).
\]
First note that by Example \ref{Example-m=conj} and by naturality the conjugation actions we have 
\begin{align*}
k\comp\zeta\comp(1_{K}\tensor s_{\psi})&= k\comp \varphi\comp (1_{K}\tensor p_{\psi})\comp (1_{K}\tensor s_{\psi})= k\comp \varphi=
k\comp c^{K,K\rtimes_{\varphi}B}\comp (1_{K}\tensor s_{\varphi})\\
&= c^{D,D}\comp (k\tensor \som{k}{\beta}) \comp (1_{K}\tensor s_{\psi})=c^{D,D}\comp (k\tensor \beta)\\
&=c^{D,D}\comp (k\tensor \som{l}{\beta})\comp(1_{K}\tensor s_{\psi}),
\end{align*}
so that Diagram~\eqref{fundlem-PUsquare} always commutes on $K\tensor B$.

Next, $k\comp\zeta\comp(1_{K}\tensor k_{\psi}))=k\comp\varphi\comp(1_{K}\tensor p_{\psi})\comp(1_{K}\tensor k_{\psi})=k\comp\varphi\comp(1_{K}\tensor 0)=0$. Hence, for the equality
\[
k\comp\zeta\comp(1_{K}\tensor k_{\psi}) = c^{D,D}\comp (k\tensor \som{l}{\beta})\comp(1_{K}\tensor k_{\psi})
\]
to hold, the morphism $c^{D,D}\comp(k\tensor l)=c_2^D\comp(\overline{k}\tensor \overline{l})\comp (k'\tensor l')
$ has to be trivial. (Here we write $k=\overline{k}\comp k'$, and similarly for $l$ and $\beta$.) Noting that $k'\tensor l'$ is a regular epimorphism by Proposition~\ref{cr-reg-epi}, we see that $c^{D,D}\comp(k\tensor l)=0$ precisely when~$[\overline{K},\overline{L}]={\Im}(c_2^D\comp(\overline{k}\tensor \overline{l}))$ is trivial.

Finally, 
\begin{align*}
k\comp\zeta\comp S_{1,2}^{K,L\rtimes_{\psi}B}\comp (1_K\tensor k_{\psi}\tensor s_{\psi})&= k\comp\varphi\comp(1_{K}\tensor p_{\psi})\comp
S_{1,2}^{K,L\rtimes_{\psi}B}\comp (1_K\tensor k_{\psi}\tensor s_{\psi})\\
&=k\comp\varphi\comp S_{1,2}^{K,B}\comp (1_K\tensor p_{\psi}\tensor p_{\psi})
\comp (1_K\tensor k_{\psi}\tensor s_{\psi}) \\
&= k\comp\varphi\comp S_{1,2}^{K,B}\comp (1_K\tensor 0\tensor 1_B)
\end{align*}
is zero, while
\begin{align*}
&c^{D,D}\comp (k\tensor \som{l}{\beta})\comp S_{1,2}^{K,L\rtimes_{\psi}B}\comp (1_K\tensor k_{\psi}\tensor s_{\psi})\\
&=c^{D,D}\comp S_{1,2}^{D,D}\comp (k\tensor \som{l}{\beta}\tensor \som{l}{\beta})\comp (1_K\tensor k_{\psi}\tensor s_{\psi})\\
&=c_3^D\comp (k\tensor l\tensor \beta)\\
&=c_3^D\comp (\overline{k}\tensor \overline{l}\tensor \overline{\beta}) \comp(k'\tensor l'\tensor \beta').
\end{align*}
As $k'\tensor l'\tensor \beta'$ is a regular epimorphism by Proposition~\ref{cr-reg-epi}, this tells us that Diagram~\eqref{fundlem-PUsquare} commutes on $K\tensor L\tensor B$ if and only if $[\overline{K},\overline{L},\overline{B}]={\Im}(c_3^D\comp (\overline{k}\tensor \overline{l}\tensor \overline{\beta}))$ is zero, which concludes the proof.
\end{proof}

\begin{theorem}\label{Theorem-SH}
The following are equivalent:\begin{enumerate}
\item the \emph{Smith is Huq} condition holds;
\item any two effective equivalence relations on a given object commute as soon as their normalisations do;
\item any two equivalence relations on a given object commute as soon as their normalisations do;
\item for all ideals $K$, $L$ of $X$ we have $[K,L,X]\leq [K,L]^{\Huq}$.
\end{enumerate}
\end{theorem}
\begin{proof}
Conditions (i) and (ii) are equivalent by definition. The equivalence between (ii) and (iii) is Remark~2.4 in~\cite{MFVdL}, but may also be obtained using Lemma~\ref{Lemma-Nelson}. Now suppose that (iii) holds and consider normal subobjects $K$ and $L$ of $X$. Divide out their Huq commutator
\[
\xymatrix{0 \ar[r] & [K,L]^{\Huq} \ar@{{ |>}->}[r] & X \ar@{-{ >>}}[r]^-{q} & Q \ar[r] & 0}
\]
and write $q(K)$, $q(L)\leq Q$ for the direct images of $K$ and $L$ along $q$. By Proposition~\ref{Proposition-Commutator-Rules}.ii we obtain a diagram
\[
\xymatrix{&& [K,L,X] \ar@{{ >}->}[d] \ar@{-{ >>}}[r] \ar@{.>}[ld] & [q(K),q(L),Q] \ar@{{ >}->}[d]\\
0 \ar[r] & [K,L]^{\Huq} \ar@{{ |>}->}[r] & X \ar@{-{ >>}}[r]_-{q} & Q \ar[r] & 0}
\]
and a factorisation of $[K,L,X]$ over $[K,L]^{\Huq}$. Indeed, $[q(K),q(L),Q]$ is zero by Theorem~\ref{Theorem-SH-1}, as $[q(K),q(L)]=q[K,L]=0$.
Finally, (iv) \implies\ (ii) is again a consequence of Theorem~\ref{Theorem-SH-1}.
\end{proof}

This at once yields a new class of examples.

\begin{example}
A \defn{nilpotent category of class $2$} is a semi-abelian category whose identity functor is \defn{quadratic}, which means that it has a trivial ternary co-smash product~\cite{CCC}. Hence, almost by definition, any such category satisfies (SH). In particular, the \emph{Smith is Huq} condition holds for modules over a square ring, and specifically for algebras over a nilpotent algebraic operad of class two~\cite{BHP}.
\end{example}

\begin{example}
If $K$, $L$ and $M$ are normal subgroups of a group $G$ then
\[
[K,L,M]= [K,[L,M]]\join [L,[M,K]]\join [M,[K,L]]
\]
by a result in~\cite{CCC}. Hence in $\Gp$ all ternary commutator words are essentially of the shape considered in Example~\ref{Example-Binary-Groups}.

This of course also gives (SH). So far it is not clear which categories allow a similar decomposition of their ternary commutators.
\end{example}

For instance, the semi-abelian variety $\Loop$ of loops and loop homomorphisms forms a counterexample. We show that it does not satisfy the \emph{Smith is Huq} condition, which also implies that this category is neither action accessible nor strongly protomodular.

\begin{example}\label{Example-Loops}
A \defn{loop} is a quasigroup with unit, an algebra
\[
(A,\cdot,\backslash,/,1)
\]
of which the multiplication $\cdot$ and the left and right division $\backslash$ and $/$ satisfy the axioms
\begin{align*}
y &= x\cdot(x\backslash y) \qquad & y = x\backslash(x\cdot y)\\
x &= (x/y)\cdot y & x = (x\cdot y)/y
\end{align*} 
and $1$ is a unit for the multiplication, $x\cdot 1 = x = 1\cdot x$. We shall sometimes write~$xy$ for~$x\cdot y$. The variety $\Loop$ of loops is semi-abelian (as mentioned for instance in~\cite{Borceux-Clementino}). Loops are ``non-associative groups'', and indeed an associative loop is the same thing as a group. It is easily seen that the abelian objects in~$\Loop$ are precisely the abelian groups---which are not to be confused with the objects in the variety of commutative loops, which have a commutative, but possibly non-associative, multiplication.

The \defn{associator} of three elements $x$, $y$, $z$ of a loop $X$ is the unique element $\ldbrack x,y,z\rdbrack$ of $X$ such that $(xy)z=\ldbrack x,y,z\rdbrack\cdot x(yz)$. Hence $\ldbrack x,y,z\rdbrack$ is equal to $(xy\cdot z)\slash(x\cdot yz)$. Given three normal subloops $K$, $L$ and $M$ of $X$, we write~$\ldbrack K,L,M\rdbrack$ for the \defn{associator subloop} of $X$ determined by $K$, $L$ and~$M$: this is the normal subloop of~$K\join L\join M$ generated by the elements $\ldbrack x,y,z\rdbrack$, where either~$(x,y,z)$ or any of its permutations is in $K\times L\times M$. 

It is clear that the object $\ldbrack K,L,M\rdbrack$ is a subloop of the ternary commutator $[K,L,M]$, as for any associator element $\ldbrack x,y,z\rdbrack$, the associators $\ldbrack 1,y,z\rdbrack$, $\ldbrack x,1,z\rdbrack$ and $\ldbrack x,y,1\rdbrack$ are trivial (Example~\ref{Example-Varieties}). 

In order to prove that the category $\Loop$ does not satisfy the \emph{Smith is Huq} condition, it suffices to give an example of a loop $X$ with an abelian normal subloop~$A$ of~$X$ such that $[A,A,X]$ is non-trivial. Then by Theorem~\ref{Theorem-SH-1} the denormalisation~$R_{A}$ of~$A$ does not Smith-commute with itself, even though $[A,A]=0$. In fact, in our example, already the associator $\ldbrack A,A,X\rdbrack$ is non-trivial. (Universal algebraists have known about the bad behaviour of commutators in the category of loops for a long time. A different example is given in~\cite[Exercise~5.10]{Freese-McKenzie}.)

We take $X$ to be the well-known (and historically important) loop of order eight occurring in relation with the hyperbolic quaternions: it is the set
\[
\{1,-1,i,-i,j,-j,k, −k\}
\]
with multiplication determined by the rules
\begin{align*}
ij &= k = - ji\\
jk &= i = - kj \qquad\qquad ii = jj = kk =1\\
ki &= j = - ik
\end{align*}
and the expected behaviour for $-1$. The subset $\{1,-1,j,-j\}$ of $L$ forms a normal subloop $A$ of index two, isomorphic to the Klein four-group $V\cong \ZZ_{2}\times \ZZ_{2}$. Now $j\cdot ji = j(-k) = - i$ while $jj\cdot i = i$, so
\[
1\neq\ldbrack j,j,i\rdbrack\in \ldbrack A,A,X\rdbrack\leq [A,A,X].
\]
\end{example}

\subsection{Decomposition of the Smith commutator}
The above Theorem~\ref{Theorem-SH-1} leads to a formula for the Smith commutator of two equivalence relations in terms of binary and ternary commutators of their normalisations: Theorem~\ref{Theorem-Smith-Commutator}.

\begin{proposition}\label{Proposition-Normality}\cite{Actions,MM-NC}
In a semi-abelian category, for $K$,~${L\leq X}$, the subobject~$K$ is normal in $K \join L$ if and only if~$[K,L]\leq K$. In particular,
\begin{enumerate}
\item $K\normal X$ if and only if $[K,X]\leq K$;
\item a morphism $f\colon X\to Y$ is proper if and only if the composite morphism $c_2^Y\comp (f\tensor 1_{Y})$ factors through $\Im (f)$.\noproof
\end{enumerate}
\end{proposition}

\begin{remark}
\cite{Actions} The characterisation (i) of normal subobjects is valid in a finitely cocomplete homological category if and only if this category is semi-abelian.
\end{remark}

\begin{lemma}[cf.\ Remark~\ref{Remark-Huq-Join}]\label{Lemma-Join}
For any $K$, $L\leq X$ in a semi-abelian category, the join $[K,L,X]\join [K,L]$ is normal in $X$.
\end{lemma}
\begin{proof}
Consider first the quotient $q$ of $X$ by $[K,L,X]$, then the direct image of~$[K,L]$ along $q$.
\[
\xymatrix{&& [K,L] \ar@{{ >}->}[d] \ar@{-{ >>}}[r] & [q(K),q(L)] \ar@{{ >}->}[d]\\
0 \ar[r] & [K,L,X] \ar@{{ |>}->}[r] & X \ar@{-{ >>}}[r]_-{q} & Q \ar[r] & 0}
\]
Note that $[K,L,X]$ is normal in $X$ by Proposition~\ref{Proposition-Commutator-Rules}.vi. To prove our claim we only need to show that the commutator $[q(K),q(L)]$ is normal in~${Q=q(X)}$. But
\[
[[q(K),q(L)],q(X)]\leq [q(K),q(L),q(X)] = q[K,L,X] = 0
\]
by Proposition~\ref{Proposition-Commutator-Rules} so that the result follows from Proposition~\ref{Proposition-Normality}.
\end{proof}

\begin{remark}\label{Remark-Smaller-M}
If we now consider $M\leq X$ such that ${K\join L\join M}$ is~$X$ then 
\[
[K,L,M]\join [K,L]=[K,L,X]\join [K,L].
\]
Indeed, freely using the rules from Proposition~\ref{Proposition-Commutator-Rules}, we see that
\begin{align*}
[K,L,K\join L \join M] \quad
= \quad & [K,L,K,L,M]\join [K,L,K,L]\join [K,L,L,M]\\
 &\join [K,L,K,M]\join [K,L,K]\join [K,L,L]\join [K,L,M] \\
\leq \quad & [K,L,M]\join [K,L]\join [K,L,M]\\
& \join [L,K,M]\join [L,K]\join [K,L]\join [K,L,M] \\
= \quad & [K,L,M]\join [K,L],
\end{align*}
while the other inclusion is obvious.
\end{remark}

\begin{remark}\label{Remark-Join}
If $K$, $L\normal X$ are such that $K \join L=X$ then $[K,L]=0$ suffices for the denormalisations $R$ of $K$ and $S$ of $L$ to commute in the Smith-sense~\cite{EverVdLRCT}. In other words, when $[K,L]$ is trivial, the ternary commutator $[K,L,X]$ is trivial as well. By Remark~\ref{Remark-Smaller-M} this also follows from
\[
[K,L,X]\join [K,L]=[K,L,0]\join [K,L]=[K,L].
\]
\end{remark}

\begin{theorem}\label{Theorem-Smith-Commutator}
In a semi-abelian category, given equivalence relations $R$ and~$S$ on $X$ with normalisations $K$, $L\normal X$, the Smith commutator $[R,S]^{\Smith}$ is the left-hand side equivalence relation
\begin{equation*}\label{Smith}
\xymatrix@!0@=10em{{([K,L,X]\join [K,L]) \rtimes_{\gamma} X} \ar@<1ex>[r]^-{\somm{0}{0}{1_{X}}} \ar@<-1ex>[r]_-{\somm{[k,l,1_{X}]}{[k,l]}{1_{X}}} & X \ar[l]|-{s_{\gamma}}}
\qquad \qquad
\xymatrix@!0@=6em{{[K,L] \rtimes_{\gamma} X} \ar@<1ex>[r]^-{\som{0}{1_{X}}} \ar@<-1ex>[r]_-{\bigsom{[k,l]}{1_{X}}} & X \ar[l]|-{s_{\gamma}}}
\end{equation*}
where $\gamma$ is the conjugation action of $X$ on $[K,L,X]\join [K,L]$. If~${K\vee L =X}$ then~$[R,S]^{\Smith}$ simplifies to the above right-hand side equivalence relation.
\end{theorem}
\begin{proof}
The equivalence relation in the statement above is the denormalisation of the normal subobject $[K,L,X]\join [K,L]$ of $X$ considered in Lemma~\ref{Lemma-Join}. By Theorem~\ref{Theorem-SH-1} it satisfies the same universal property as $[R,S]^{\Smith}$, hence the two coincide. The further refinement is just Remark~\ref{Remark-Join}.
\end{proof}

\subsection{An application to homology}\label{Subsection-Homology}
One situation where expressing the Smith commutator in terms of tensor products yields immediate results is in semi-abelian homology. For instance, according to~\cite{EGVdL} the Hopf formula for the third homology object $\H_{3}(Z,\ab)$ of an object $Z$ with coefficients in the abelianisation functor
\[
\ab\colon\A\to\Ab(\A)\colon A\mapsto A/[A,A]^{\Huq}
\]
depends on a characterisation of the double central extensions in $\A$.
Such a characterisation was given in~\cite{RVdL} in terms of the Smith commutator: a double extension such as~\eqref{Double-Extension} below is central if and only if
\[
[R,S]^{\Smith}=\Delta_{X}=[R\meet S,\nabla_{X}]^{\Smith}.
\]
Here $\nabla_{X}$ is the largest equivalence relation on $X$, the denormalisation of $1_{X}$, and $R$ and $S$ are the kernel relations of $d$ and $c$, respectively. If (SH) holds then this condition may be reformulated in terms of the Huq commutator, and when $\A$ has enough projectives this makes it possible to express $\H_{3}(Z,\ab)$ as a quotient of commutators. So far, however, it was unclear how to obtain a similar explicit formula in categories that do not satisfy (SH).

Recall that a \defn{double extension} in a semi-abelian category $\A$ is a pushout square~\eqref{Double-Extension} of which all arrows are regular epimorphisms~\cite{EGVdL}. A \defn{double presentation} of an object $Z$ is a double extension such as~\eqref{Double-Extension} in which the objects $X$, $D$ and $C$ are (regular epi)-projective. Higher extensions were introduced in~\cite{EGVdL} following~\cite{Janelidze:Double} and~\cite{Janelidze-Kelly} in order to capture the concept of \emph{higher centrality} which is useful in the study of semi-abelian (co)ho\-mo\-lo\-gy: see, for instance, the articles~\cite{EverHopf, EGVdL, RVdL2}.

\begin{proposition}\label{Characterisation-Double-Central-Extensions}
Given a double extension
\begin{equation}\label{Double-Extension}
\vcenter{\xymatrix{X \ar[r]^-{c} \ar[d]_-{d} & C \ar[d]^-{g}\\
D \ar[r]_-{f} & Z}}
\end{equation}
in a semi-abelian category, write $K=\Ker(c)$ and $L=\Ker(d)$. Then~\eqref{Double-Extension} is central if and only if
\[
[K,L,X]=[K,L]=[K\meet L,X]=0.
\]
\end{proposition}
\begin{proof}
Via Theorem~\ref{Theorem-SH-1} this is an immediate consequence of~\cite[Theorem~2.8]{RVdL}.
\end{proof}

\begin{theorem}\label{Theorem-H3}
Let $\A$ be a semi-abelian category with enough projectives. Let~$Z$ be an object in $\A$ and~\eqref{Double-Extension} a double presentation of $Z$ with $K=\Ker(c)$ and ${L=\Ker(d)}$. Then
\[
\H_{3}(Z,\ab)\cong\frac{K\meet L\meet [X,X]}{[K,L,X]\join [K,L]\join [K\meet L,X]}.
\]
If $\A$ is monadic over~$\Set$ then these homology groups are comonadic Barr--Beck ho\-mo\-lo\-gy~\cite{Barr-Beck} with respect to the canonical comonad on $\A$.
\end{theorem}
\begin{proof}
This follows from Proposition~\ref{Characterisation-Double-Central-Extensions} and the main result of~\cite{EverHopf}; see also~\cite{EGVdL}. Note that by Lemma~\ref{Lemma-Join} and Proposition~\ref{Proposition-Commutator-Rules}.vi, the denominator is indeed normal in~$X$ so that the formula makes sense.
\end{proof}

\begin{remark}
Note that in the groups case~\cite{Brown-Ellis} the ternary commutator in the above formula is invisible, as it is contained in $[K,L]$.
\end{remark}

\section{Internal crossed modules}\label{Section-Crossed-Modules}
Internal crossed modules were introduced in the article~\cite{Janelidze}. Here we study them from the viewpoint of co-smash products. We obtain a new characterisation which involves a higher coherence condition. This condition does not appear in any of the usual categories where crossed modules have been considered so far, such as groups, Lie algebras and associative algebras: it expresses the property (SH) needed to extend a star-multiplication to an internal category structure in arbitrary semi-abelian categories, or even finitely cocomplete homological ones---see~\cite{Janelidze, MFVdL}.

\subsection{Internal categories}
The analysis of the \emph{Smith is Huq} condition in terms of higher-order commutators yields new conditions for an internal reflexive graph to be an internal category (or, equivalently, an internal groupoid); cf.~\cite{Loday} for the equivalence between (i) and (ii) in the case of groups. 

\begin{theorem}\label{Theorem-Internal-Categories}
Consider an internal reflexive graph $(R,G,d,c,e)$ in a finitely cocomplete homological category.
\begin{equation*}\label{RG}
\xymatrix@!0@=5em{R \ar@<1ex>[r]^-{d} \ar@<-1ex>[r]_-{c} & G \ar[l]|-{e}}
\qquad\qquad
d\comp e = c\comp e = 1_{G}
\end{equation*}
The following are equivalent:
\begin{enumerate}
\item $(R,G,d,c,e)$ is an internal category;
\item $[\Ker(d),\Ker(c)]=0=[\Ker(d),\Ker(c),R]$;
\item $[\Ker(d),\Ker(c)]=0=[\Ker(d),\Ker(c),\Im(e)]$;
\item the morphism $c^{A,R}\colon{A\tensor R\to A}$ induced by the conjugation action of $R$ on~$A=\Ker(d)$ factors through $1_{A}\tensor c\colon{A\tensor R\to A\tensor G}$;
\item $c^{A,R}=(e\comp c)^{*}(c^{A,R})$.
\end{enumerate}
\end{theorem}
\begin{proof}
Theorem~\ref{Theorem-SH-1} implies that (i) and (ii) are equivalent, because the given reflexive graph is a groupoid if and only if the kernel pairs of~$d$ and~$c$ Smith-commute~\cite{Pedicchio}. It is clear that (ii) implies (iii), while the equivalence between (i) and (iii) may be obtained via Lemma~\ref{Lemma-Nelson}. In fact, (ii) also follows from (iii) by a direct commutator calculation using Proposition~\ref{Proposition-Commutator-Rules}, since~${R=A\join \Im(e)}$. 

The equivalence between (iii) and (iv) is a consequence of Proposition~\ref{Split-Exact}. Finally, if $c^{A,R}=c^{*}(\varphi)$ then
\[
e^{*}(c^{A,R})=e^{*}(c^{*}(\varphi))=(c\comp e)^{*}(\varphi)=\varphi,
\]
so that $c^{A,R}=c^{*}(e^{*}(c^{A,R}))=(e\comp c)^{*}(c^{A,R})$.
\end{proof}

Condition (ii) on commuting kernels says that a reflexive graph $(R,G,d,c,e)$ with a multiplication $m\colon{\Ker(d)\times\Ker(c)\to R}$ defined locally around $0$ as in
\[
\vcenter{\xymatrix@1@!0@R=2.4495em@C=1.4142em{& {0} \ar[ld]_-{\beta}\\
{\cdot} && {\cdot} \ar[lu]_-{\alpha} \ar@{.>}[ll]^-{\gamma}}}
\qquad\qquad
m(\beta,\alpha)=\gamma
\]
such that $m\comp\langle 1_{\Ker(d)},0\rangle=\ker (d)$ and $m\comp\langle 0,1_{\Ker(c)}\rangle=\ker (c)$ admits a globally defined multiplication (that is, an internal category structure) if and only if the obstruction $[\Ker(d),\Ker(c),R]$ vanishes. Similar ``local to global'' properties were studied in~\cite{MM, MFVdL} after they appeared naturally in~\cite{Janelidze}. Since both are relevant in what follows, we briefly recall their definition; see~\cite{MM, MFVdL} and Remark~\ref{Remark-Peiffer} for more details and a proof that the structures are equivalent.

Consider a reflexive graph $(R,G,d,c,e)$ and the pullback 
\[
\vcenter{\xymatrix@C=3em{R\times_{G} \Ker(d) \ar@{}[rd]|<<{\pullback} \ar[d]_-{\pi_{R}} \ar[r]^-{\pi_{\Ker(d)}} & \Ker(d) \ar[d]^-{\del=c\circ\ker(d)}\\
R \ar[r]_-{d} & G.}}
\]
The reflexive graph $(R,G,d,c,e)$ is a \defn{star-multiplicative graph}~\cite{Janelidze} when there is a (necessarily unique) morphism $\varsigma\colon{R\times_{G} \Ker(d)\to \Ker(d)}$ such that the conditions $\varsigma \comp\langle \ker(d),0\rangle=1_{\Ker(d)}$ and $\varsigma \langle e\comp\del,1_{\Ker(d)}\rangle=1_{\Ker(d)}$ hold.
\[
\vcenter{\xymatrix@1@!0@R=2.4495em@C=1.4142em{& {\cdot} \ar[ld]_-{\beta}\\
{\cdot} && {0} \ar[lu]_-{\alpha} \ar@{.>}[ll]^-{\gamma}}}
\qquad
\zeta(\beta,\alpha)=\gamma
\qquad\qquad\qquad\qquad
\vcenter{\xymatrix@1@!0@R=2.4495em@C=1.4142em{& {0} \ar[ld]_-{\beta} \ar[rd]^-{\alpha}\\
{\cdot} && {\cdot} \ar@{.>}[ll]^-{\gamma}}}
\qquad
\omega(\beta,\alpha)=\gamma
\]
It is said to be a \defn{Peiffer graph}~\cite{MM} when there is a (necessarily unique) morphism $\omega\colon{\Ker(d)\times \Ker(d)\to R}$ such that $\omega\comp\langle 1_{\Ker(d)},0\rangle=\ker(d)$ and $\omega\comp\langle 1_{\Ker(d)},1_{\Ker(d)}\rangle =e\comp c\comp\ker(d)$.

\subsection{Precrossed modules and crossed modules}
A precrossed module is a normalisation of a reflexive graph, while a crossed module is a normalisation of an internal groupoid. We describe these structures in terms of co-smash products.

A \defn{precrossed module} in a finitely cocomplete homological category~$\A$ may be encoded as a quadruple $(G,A,\mu ,\del)$ where~$G$ and~$A$ are objects in $\A$, $\mu \colon{A\tensor G\to A}$ determines an action of $G$ on $A$, and $\del\colon{A \to G}$ is a $G$-equivariant morphism with respect to the action determined by $\mu $ and the conjugation action of $G$ on itself, respectively. In other words, the diagram
\begin{equation}\label{precrmod-dia}
\vcenter{\xymatrix{
A\tensor G \ar[r]^-\mu \ar[d]_{\del\tensor 1_{G}} & A\ar[d]^{\del}\\
G\tensor G\ar[r]_-{c^{G,G}} & G
}}
\end{equation}
commutes. Together with the obvious morphisms, the precrossed modules in $\A$ form a category~$\PXMod (\A)$.

\begin{proposition}\label{pxm-refl-graph}
The category $\PXMod(\A)$ is equivalent to $\RG(\A)$. 
\end{proposition}
\begin{proof}
This is an extension of the equivalence between actions and split epimorphisms. Given a precrossed module $(G,A,\mu ,\del)$, the action $\mu $ corresponds to a split exact sequence
\begin{equation*}\label{RG+ker}
\xymatrix{0 \ar[r] & A \ar@{{ |>}->}[r]^-{\ker (d)} & R \ar@{-{ >>}}@<1ex>[r]^-{d} \ar@{.{ >>}}@<-1ex>[r]_-{c} & G \ar[r] \ar@{{ >}->}[l]|-{e} & 0}
\end{equation*}
where $R=A\rtimes_{\mu }G$. Proposition~\ref{PU} gives a unique morphism $c\colon{R\to G}$ such that~${\del=c\comp\ker (d)}$ and $c\comp e=1_{G}$ precisely when~\eqref{precrmod-dia} commutes.
\end{proof}

\begin{definition}\label{definition-xmod}
A precrossed module $(G,A,\mu ,\del)$ is a \defn{crossed module} if its associated reflexive graph is an internal category. This gives us the full reflective~\cite{Pedicchio} subcategory $\XMod(\A)$ of $\PXMod(\A)$. 
\end{definition}

Janelidze analysed this concept of crossed module using internal actions in semi-abelian categories~\cite{Janelidze}. Here the actions are treated differently, and thus we obtain a different characterisation, which is moreover valid in a non-exact context:

\begin{theorem}\label{Theorem-Characterisation-XMod}
A precrossed module $(G,A,\mu ,\del)$ in a finitely cocomplete homological category is a crossed module if and only if it satisfies the following two additional conditions:
\begin{enumerate}
\item the conjugation action of $A$ on itself coincides with the pullback of~$\mu$ along~$\del$, that is, $c^{A,A}=\del^{*}(\mu)$ so that the diagram
\begin{equation}\label{crmod-dia}
\vcenter{\xymatrix{
A\tensor A \ar[r]^-{c^{A,A}} \ar[d]_-{1_{A}\tensor \del} & A \ar@{=}[d] \\
A\tensor G \ar[r]_-{\mu} & A
}}
\end{equation}
commutes;
\item the diagram
\begin{equation}\label{crmod-coh-dia}
\vcenter{\xymatrix{
A\tensor A\tensor G \ar[r]^-{\mu_{2,1}} \ar[d]_-{1_{A}\tensor \del\tensor 1_{G}} & A \ar@{=}[d] \\
A\tensor G\tensor G \ar[r]_-{\mu_{1,2}} & A}}
\end{equation}
commutes.
\end{enumerate}
\end{theorem}
\begin{proof}
Using Lemma~\ref{Tensor vs sum}, we decompose the object $R$ in such a way that the fifth condition of Theorem~\ref{Theorem-Internal-Categories} falls apart in three distinct statements. One of those is the commutativity of~\eqref{crmod-dia}, a second one is the commutativity of~\eqref{crmod-coh-dia}, and a third one is trivially satisfied.

Indeed, $R=A\rtimes_{\mu}G$, so that we may consider the pair of parallel morphisms
\[
\resizebox{\textwidth}{!}
{\xymatrix{((A\tensor A\tensor G)\rtimes (A\tensor A))\rtimes (A\tensor G) \ar@{-{ >>}}[r] & A\tensor (A+G) \ar@{-{ >>}}[r]^-{1_{A}\tensor q} & A\tensor (A\rtimes_{\mu}G) \ar@<.5ex>[rr]^-{c^{A,R}} \ar@<-.5ex>[rr]_-{(e\circ c)^{*}(c^{A,R})} && A.}}
\]
On $A\tensor G$ these morphisms coincide, as $q\comp \iota_{G}=e\colon G\to A\rtimes_{\mu}G=R$ by definition of $e$, and
\begin{align*}
(e\comp c)^{*}(c^{A,R})\comp (1_{A}\tensor e) &= e^{*}((e\comp c)^{*}(c^{A,R}))= (e\comp c\comp e)^{*}(c^{A,R})\\
&= e^{*}(c^{A,R}) = c^{A,R}\comp (1_{A}\tensor e).
\end{align*}
On $A\tensor A$ they coincide if and only if the diagram~\eqref{crmod-dia} commutes. To see this, recall that~$q=\som{\ker(d)}{e}\colon{A+G\to A\rtimes_{\mu}G=R}$, so that $q\comp \iota_{A}$ is the monomorphism~$\ker(d)\colon {A\to R}$. Then
\[
\ker(d)\comp c^{A,R}\comp (1_{A}\tensor \ker (d)) = \ker(d)\comp c^{A,A}
\]
by naturality of conjugation actions (Proposition~\ref{conjnormsub}), and
\begin{align*}
\ker(d)\comp (e\comp c)^{*}(c^{A,R})\comp (1_{A}\tensor \ker (d)) 
& = 
\ker(d)\comp c^{A,R}\comp (1_{A}\tensor (e\comp c))\comp(1_{A}\tensor \ker (d))\\
& = \ker(d)\comp c^{A,R}\comp (1_{A}\tensor e)\comp (1_{A}\tensor (c\comp\ker (d)))\\
& = \ker(d)\comp \mu\comp (1_{A}\tensor \del).
\end{align*}
Hence $c^{A,A} = \mu\comp (1_{A}\tensor \del)$ if and only if $c^{A,R}$ and $(e\comp c)^{*}(c^{A,R})$ coincide on $A\tensor A$.

Similarly, $c^{A,R}$ and $(e\comp c)^{*}(c^{A,R})$ coincide on $A\tensor A\tensor G$ precisely when~\eqref{crmod-coh-dia} commutes. For a proof, consider the commutative diagrams
\[
\xymatrix@!0@R=4em@C=9em{A\tensor A\tensor G \ar[d]_{\iota_{A,G}^{A\tensor (-)}} \ar[rd]^-{\iota_{A,A,G}} \ar[r]^{1_{A}\tensor 1_{A}\tensor e} & A\tensor A\tensor R \ar[rd]^-{\iota_{A,A,R}} \\ 
 A\tensor (A+G) \ar[r]_-{\iota_{A,A+G}} \ar[d]_-{1_{A}\tensor q} & A+A+G \ar[d]_-{1_{A}+q} \ar[r]^-{1_{A}+1_{A}+e} & A+A+R \ar@/^1em/[ld]|-{\somm{\iota_{A}}{\iota_{R}\circ\ker(d)}{\iota_{R}}} \ar@/^2em/[ddl]|-{\somm{\ker(d)}{\ker(d)}{1_{R}}} \ar[dd]^-{\ker(d)+\ker(d)+1_{R}} \\
A\tensor R \ar[d]_-{c^{A,R}} \ar^-{\iota_{A,R}}[r] & A+R \ar[d]_-{\som{\ker (d)}{1_{R}}} \\
A \ar@{{ |>}->}[r]_-{\ker (d)} & R & R+R+R \ar[l]^-{\nabla_{R}^{3}}}
\]
and
\[
\xymatrix@!0@R=4em@C=8em{
A\tensor A\tensor G \ar[r]^-{1_{A}\tensor 1_{A}\tensor e} \ar[d]^{S^{A,G}_{2,1}} \ar@/_2em/[dd]_-{\mu_{2,1}} & A\tensor A\tensor R \ar@/^1em/[rrd]^-{\iota_{A,A,R}} \ar[d]^-{S^{A,R}_{2,1}} \\
A\tensor G \ar[d]^-{\mu} \ar[r]^-{1_{A}\tensor e} & A\tensor R \ar[d]^-{c^{A,R}} \ar^-{\iota_{A,R}}[r] & A+R \ar[d]_-{\som{\ker (d)}{1_{R}}} & A+A+R \ar[d]^-{\ker(d)+\ker(d)+1_{R}} \ar[l]_-{\nabla_{A}+1_{R}}\\
A \ar@{=}[r] & A \ar@{{ |>}->}[r]_-{\ker (d)} & R & R+R+R \ar[l]^-{\nabla_{R}^{3}}}
\]
which show that $\mu_{2,1}=c^{A,R}\comp (1_{A}\tensor q)\comp \iota^{A\tensor (-)}_{A,G}$. Similar diagrams show that
\[
\mu_{1,2}\comp (1_{A}\tensor \del\tensor 1_{G})=(e\comp c)^{*}(c^{A,R})\comp (1\tensor q)\comp \iota^{A\tensor (-)}_{A,G},
\]
and these two equalities together are precisely what we need to prove our claim.
\end{proof}

Alternatively, in this proof we could have used Sequence~\eqref{Sequence-Rightex-Split-cross} as in the proof of Lemma~\ref{Lemma-Nelson}.

\begin{remark}\label{Remark-Peiffer}
Condition (i) could be called the \defn{Peiffer condition}. It means that the reflexive graph induced by $(G,A,\mu ,\del)$ is a Peiffer graph: the commutativity of~\eqref{crmod-dia} gives us a morphism of split short exact sequences
\[
\xymatrix@!0@R=3em@C=5em{0 \ar[r] & A \ar@{=}[d] \ar@{{ |>}->}[r]^-{\langle1_{A},0\rangle} & A\times A \ar@{-{ >>}}@<.5ex>[r]^-{\pi_{2}} \ar[d]_-{\omega} & A \ar[d]^-{\del} \ar[r] \ar@{{ >}->}@<.5ex>[l]^-{\langle1_{A},1_{A}\rangle} & 0\\
0 \ar[r] & A \ar@{{ |>}->}[r]_-{ \ker(d)} & R \ar@{-{ >>}}@<.5ex>[r]^-{d} & G \ar[r] \ar@{{ >}->}@<.5ex>[l]^-{e} & 0}
\]
as in Example~\ref{Example-Conjugation}. The conditions $ \ker(d)=\omega\comp \langle1_{A},0\rangle$ and $e\comp\del=\omega\comp\langle1_{A},1_{A}\rangle$ tell us that $\omega$ is a Peiffer structure on $(R,G,d,c,e)$. By Proposition~3.7 in~\cite{MFVdL} this is equivalent to the reflexive graph being star-multiplicative in the sense of~\cite{Janelidze}, or---when $\A$ is semi-abelian---the condition that $\ker (d)$ and $\ker (c)$ commute. 

The star-multiplication on $(R,G,d,c,e)$ may also be obtained directly from the commutativity of~\eqref{crmod-dia}. Indeed, via the co-universal property of semi-direct products (Proposition~\ref{PU}) we see that the needed morphism
\[
\zeta\colon{A\rtimes_{\del^{*}(\mu)} A=R\times_{G}A\to A}
\]
exists if and only if $\del^{*}(\mu)=c^{A,A}$.

Hence a semi-abelian category satisfies (SH) if and only if the coherence condition~(ii) always comes for free: every precrossed module that satisfies the Peiffer condition is a crossed module.

In a non-exact context this is not quite true. As explained in the last paragraph of~\cite{MFVdL}, in order that (SH) be equivalent to the condition ``all star-multiplications come from internal category structures'', a slight strengthening of the definitions of star-multiplicative graph and of Peiffer graph imposes itself. Thus asking that (ii) always follows from (i) in a finitely cocomplete homological category seems formally stronger than assuming (SH), as the Peiffer condition~(i) only gives a ``weak'' star-multiplication. 
\end{remark}

\begin{examples}\label{crmod-gr-alg}
In the case of augmented (= non-unitary) associative algebras we recover the definition of crossed modules due to Dedecker and Lue~\cite{Dedecker-Lue, Lue-Crossed} and Baues~\cite{Baues-Minian}, and in the case of Lie algebras the one considered by Kassel and Loday~\cite{Kassel-Loday}. Note, however, that in all these categories the coherence condition~\eqref{crmod-coh-dia} comes for free, because all of them have the \emph{Smith is Huq} property. So the description in terms of star-multiplicative graphs of~\cite{Janelidze} would have given the same result.
\end{examples}

\section{Beck modules}\label{Section-Beck}
As explained in~\cite{Bourn-Janelidze:Torsors}, there is a subtle difference between the concept of \defn{extension with abelian kernel}---any short exact sequence
\begin{equation}\label{exte}
\xymatrix{0 \ar[r] & A \ar@{{ |>}->}[r]^-{a} & X \ar@{-{ >>}}[r]^-{p} & G \ar[r] & 0}
\end{equation}
where the kernel $A$ is abelian---and the notion of \defn{abelian extension}, a regular epimorphism $p\colon{X\to G}$ which is an abelian object in the slice category $(\A\comma G)$. Since ``abelian object'' here means that $p$ admits an internal Mal'tsev operation, this amounts to the condition $[R,R]^{\Smith}=\Delta_{X}$ where $R$ is the kernel relation of $p$. It is clear that the difference between the two concepts is again an instance of the \emph{Smith is Huq} condition.

While abelian extensions are abelian objects in a slice category $(\A\comma G)$, \emph{Beck modules}~\cite{Beck, Barr-Beck} are abelian groups in $(\A\comma G)$ or, equivalently, abelian objects in the category of points $\Pt_{G}(\A)$. Hence from~\cite{Bourn-Janelidze:Semidirect, Bourn-Janelidze:Torsors} it follows immediately that modules are \emph{abelian actions}. In the present section we obtain a further refinement in terms of (higher-order) tensor products, valid in a context where \emph{Smith is Huq} need not hold.

Given an object $G$ of a finitely cocomplete homological category $\A$, a \defn{$G$-module} or \defn{Beck module over $G$} is an abelian group in the slice category~$(\A\comma G)$. Thus a $G$-module~$(p,m,s)$ consists of a morphism~$p\colon {X\to G}$ in $\A$, equipped with a multiplication~$m$ and a unit~$s$ as in the commutative triangles
\[
\vcenter{\xymatrix@!0@R=3em@C=2em{X\times_{G}X \ar[rr]^-{m} \ar[rd]_-{p_{\times}} && X \ar[ld]^-{p} \\ & G}}
\qquad
\qquad
\vcenter{\xymatrix@!0@R=3em@C=2em{G \ar[rr]^-{s} \ar@{=}[rd] && X \ar[ld]^-{p} \\ & G}}
\]
satisfying the usual axioms. (Here we write $X\times_{G}X$ for the kernel pair of~$p$, and we put $p_{\times}=p\comp m=p\comp \pi_{1}=p\comp \pi_{2}$.) In particular we obtain a split short exact sequence
\begin{equation}\label{sses}
\vcenter{\xymatrix{0 \ar[r] & A \ar@{{ |>}->}[r]^-{\ker(p)} & X \ar@<.5ex>@{ >>}[r]^-{p} & G \ar[r] \ar@{{ >}->}@<.5ex>[l]^-{s} & 0}}
\end{equation}
where $A$ is an abelian object in $\A$ and $p$ is split by $s$. Furthermore, since as an abelian extension it carries an internal Mal'tsev operation, the morphism $p$ satisfies $[X\times_{G}X,X\times_{G}X]^{\Smith}=\Delta_{X}$. Conversely, given the splitting $s$ of~$p$, this latter condition makes it possible to recover the multiplication $m$. We write~$\Mod_{G}(\A)$ for the category~${\Ab(\A\comma G)=\Mal(\Pt_{G}(\A))}$ of $G$-modules in $\A$. 

\begin{examples}\cite{Beck}
In the category $\Gp$, a Beck module over $G$ is the same thing as a classical module over the group-ring $\ZZ G$. When $\A$ is an additive category, the kernel functor determines an equivalence $\Mod_{G}(\A)\simeq \A$. In the category $\CAlg_{\KK}$ of commutative (non-unitary) algebras over a commutative ring $\KK$, a Beck module over~$G$ is a~$G$-module with a trivial multiplication. It is worth considering this latter example more in detail.

An internal $G$-action on an object $A$ is a morphism $\xi\colon{G\flat A\to A}$. Now here, $G\flat A\cong A+(A\tensor_{K} G)$ and $\xi\comp \iota_{A}=1_{A}$. So the restriction $\psi\colon{A\tensor_{K} G\to A}$ of $\xi$ to the tensor product $A\tensor_{K} G$ is nothing but the usual presentation of a $G$-module structure on $A$.
\end{examples}

\begin{theorem}\label{Theorem-Beck}
Let $A$ be an abelian object endowed with a $G$-action determined by $\psi\colon A\tensor G \to A$. Then the following are equivalent:
\begin{enumerate}
\item $(A,\psi)$ is a $G$-module;
\item $(G,A,\psi,0)$ is a crossed module;
\item $\psi_{2,1}\colon{A\tensor A\tensor G\to A}$ is trivial.
\end{enumerate}
\end{theorem}
\begin{proof}
Let~\eqref{sses} be the split short exact sequence induced by $\psi$. Then $(A,\psi)$ is a $G$-module if and only if the reflexive graph
\[
\xymatrix@!0@=5em{X \ar@<1ex>[r]^-{p} \ar@<-1ex>[r]_-{p} & G \ar[l]|-{s}}
\]
is an internal category. Since $p\comp\ker(p)=0$ this proves (i)~\equiv~(ii). 

Since $A$ is abelian, already $[A,A]=0$. So Theorem~\ref{Theorem-Characterisation-XMod} tells us that Condition~(ii) holds precisely when $\psi_{2,1}=\psi_{1,2}\comp(1_{A}\tensor 0\tensor 1_{G})=0$, that is, when (iii) holds.
\end{proof}

\begin{remark}
Condition (iii) is equivalent with requiring that $\psi_{p,q}=0$ for all~${p\geq 2}$ since these morphisms $\psi_{p,q}$ clearly factor through $\psi_{2,1}$.
\end{remark}

\begin{example}\label{Example-Loops-Semidirect}
The situation considered in Example~\ref{Example-Loops} is actually a loop action of the cyclic group of order two $\ZZ_{2}$ on the Klein four-group $V\cong A$ which is not a module structure. Indeed, the short exact sequence
\[
\vcenter{\xymatrix{0 \ar[r] & A \ar@{{ |>}->}[r] & X \ar@<.5ex>@{ >>}[r] & \{1,i\} \ar[r] \ar@{{ >}.>}@<.5ex>[l] & 0}}
\]
is split by the inclusion of $\ZZ_{2} \cong\{1,i\}$ in~$X$. (But the subloop $\{1,i\}$ is not normal in $X$, as $ij\cdot j=kj=-i\not \in \{1,i\}$ although $1j\cdot j=1$.) Hence $X\cong V\rtimes_{\psi}\ZZ_{2}$ for some action $\psi\colon{V\tensor \ZZ_{2}\to V}$ in the category of loops. Now~$(V,\psi)$ cannot be a~$\ZZ_{2}$-module, as we know that~$[R_{A},R_{A}]^{\Smith}\neq \Delta_{X}$; so $\psi_{2,1}$ must be non-trivial---and indeed, $\psi_{2,1}\ldbrack j,j,i\rdbrack = -1$. 
\end{example}

\begin{example}\label{Example-Module-in-Varieties}
In a semi-abelian variety of algebras $\V$, consider an abelian object~$A$ and a $G$-action determined by $\psi\colon A\tensor G \to A$. Then the coherence condition~${\psi_{2,1}=0}$ which must hold for $\psi$ to induce a module structure may be expressed as follows (cf.\ Example~\ref{Example-Varieties}):
\begin{gather*}
\begin{cases}
t(a_{1},\dots,a_{k},a_{k+1},\dots,a_{k+l},0,\dots,0)=0 & \text{in $A+A$}\\
t(a_{1},\dots,a_{k},0,\dots,0,g_{1},\dots,g_{m})=0 & \text{in $A+G$}\\
t(0,\dots,0,a_{k+1},\dots,a_{k+l},g_{1},\dots,g_{m})=0 & \text{in $A+G$}
\end{cases}\\
\Rightarrow \\
\psi(t(a_{1},\dots,a_{k+l},g_{1},\dots,g_{m}))=0,
\end{gather*}
for any term $t$ of arity $k+l+m$ in the theory of $\V$ and all $a_{1}$, \dots, $a_{k+l}\in A$ and $g_{1}$, \dots, $g_{m}\in G$. We believe this is a basic condition; certainly it is of the same level of complexity as for instance the characterisation of ideals due to Ursini~\cite{Ursini2}, valid in semi-abelian varieties~\cite{Janelidze-Marki-Tholen-Ursini}.
\end{example}

\begin{lemma}\label{Lemma-abker}
Consider a short exact sequence~\eqref{exte}. If $p$ is split by $s$ then the conjugation action of $X$ on $A$ admits a factorisation $c^{A,X}=\psi_{p}\comp (1_{A}\tensor p)$ if and only if $A$ is abelian and $c^{A,X}_{2,1}\comp (1_{A}\tensor 1_{A}\tensor s)=0$.
\end{lemma}
\begin{proof}
By Proposition~\ref{Split-Exact}, the morphism $c^{A,X}$ factors through $1_{A}\tensor p$ when
\[
c^{A,X}\comp S_{1,2}^{A,X}\comp(1_{A}\tensor a\tensor s)
\qquad\text{and}\qquad
c^{A,X}\comp(1_{A}\tensor a)
\]
are trivial. But $c^{A,X}\comp(1_{A}\tensor a)=a\comp c^{A,A}$ by naturality of the conjugation action, and~${c^{A,X}_{1,2}\comp(1_{A}\tensor a\tensor s) = c^{A,X}_{2,1}}\comp (1_{A}\tensor 1_{A}\tensor s)$ by Lemma~\ref{c21=c12}.
\end{proof}

\begin{theorem}\label{Theorem-Beck-2}
Let $A$ be an abelian object endowed with a $G$-action determined by $\psi\colon A\tensor G \to A$. Then $(A,\psi)$ is a $G$-module if and only if the conjugation action of $A\rtimes_{\psi} G$ on $A$ factors through the given $G$-action on~$A$ via the projection $p_{\psi}\colon{A\rtimes_{\psi} G \to G}$. In other words, 
\[
c^{A,A\rtimes_{\psi} G} = \psi\comp (1_{A}\tensor p_{\psi})= p_{\psi}^*(\psi).
\]
\end{theorem}
\begin{proof}
We pass via Condition~(iii) in Theorem~\ref{Theorem-Beck}. Recall that ${X=A\rtimes_{\psi}G}$. Applying Lemma~\ref{Lemma-abker} to the split extension
\begin{equation*}\label{sdpsequ}
\xymatrix@=3em{0 \ar[r] & A \ar@{{ |>}->}[r]^-{k_{\psi}} & A\rtimes_{\psi}G \ar@{-{ >>}}@<.5ex>[r]^-{p_{\psi}} & G \ar[r] \ar@{{ >}->}@<.5ex>[l]^-{s_{\psi}} & 0}
\end{equation*}
shows that $c^{A,X}\colon A\tensor X\to A$ factors through the morphism $1_{A}\tensor p_{\psi}$ precisely when $c^{A,X}_{2,1}\comp(1_{A}\tensor 1_{A}\tensor s_{\psi})=0$. However,
\begin{align*}
c^{A,X}_{2,1}\comp(1_{A}\tensor 1_{A}\tensor s_{\psi}) &= c^{A,X}\comp S_{2,1}^{A,X}\comp(1_{A}\tensor 1_{A}\tensor s_{\psi})\\
&= c^{A,X}\comp(1_{A}\tensor s_{\psi})\comp S_{2,1}^{A,G}\\
&= \psi\comp S_{2,1}^{A,G} = \psi_{2,1}.
\end{align*}
Now suppose that $c^{A,X}$ does factor as a composite morphism $\overline{c}\comp (1_{A}\tensor p_{\psi})$; then $\overline{c} = \overline{c}\comp (1_{A}\tensor p_{\psi})\comp(1_{A}\tensor s_{\psi}) = c^{A,X}\comp(1_{A}\tensor s_{\psi}) = \psi$, which proves our claim.
\end{proof}

\section*{Acknowledgements}
Many thanks to Marino Gran for his comments on our work. Thanks also to Nelson Martins--Ferreira for explaining us his admissibility condition.


\end{document}